\documentclass[ejs,preprint]{imsart}
\arxiv{math.ST/1312.2968}

\usepackage{hyperref}
\usepackage{natbib}
\usepackage{bbm}
\usepackage{booktabs}
\usepackage{framed}
\newcommand{\ra}[1]{\renewcommand{\arraystretch}{#1}}
\usepackage[x11names,svgnames]{xcolor}
\usepackage{algorithmic}
\usepackage{algorithm}

\newlength{\minipagewidth}
\newcommand{\beqN}{\begin{equation*}}  
\newcommand{\eeqN}{\end{equation*}}

\def\text#1{\hbox{\rm#1}}

\usepackage{amsthm,amsmath,amssymb}
\usepackage{natbib}
\usepackage{verbatim}
\usepackage{epsfig}
\usepackage{graphicx} 
\usepackage{enumerate}

\startlocaldefs
\newtheorem{theorem}{Theorem}[section]

\newtheorem{lemma}{Lemma}[section]
\newtheorem{corollary}{Corollary}[section]

\newtheorem{definition}{Definition}[section]

\newcommand{\beq}{\begin{equation}}  
\newcommand{\eeq}{\end{equation}}

\newcommand{\ben}{\begin{enumerate}}  
\newcommand{\een}{\end{enumerate}}
\newcommand{\bed}{\begin{itemize}}  
\newcommand{\eed}{\end{itemize}}

\newcommand {\BB} {{\mathcal B}}

\newcommand {\ee} {{\mathbb E}}
\newcommand {\FF} {{\mathcal F}}

\newcommand {\II}{{\mathcal I}}

\newcommand {\one} {{\mathbbm{1}}}
\newcommand {\PP} {{\mathcal P}}
\newcommand {\pp} {{\mathbb P}}

\newcommand {\rr} {{\mathbb R}}

\newcommand {\set}{{\mathcal{S}}}
\endlocaldefs

\begin{document}
\begin{frontmatter}

\title{Adaptive confidence intervals for the tail coefficient in a wide second order class of Pareto models}%
\runtitle{Uniform and adaptive confidence interval in the Pareto model}

\begin{aug}
  \author{\fnms{Alexandra}  \snm{Carpentier}
\thanksref{e1}
\ead[label=e1, mark]{a.carpentier@statslab.cam.ac.uk}}
  \and
  \author{\fnms{Arlene K. H.} \snm{Kim}
\thanksref{e2}
\ead[label=e2,mark]{a.kim@statslab.cam.ac.uk}}

  \runauthor{Carpentier and Kim}

  \affiliation{University of Cambridge and University of Cambridge}

  \address{Statistical Laboratory, Centre for Mathematical Sciences, Wilberforce Road, University of
Cambridge, UK, CB3 0WB. \printead{e1,e2}}

\end{aug}

\begin{abstract}
We study the problem of constructing uniform and adaptive confidence intervals for the tail coefficient in a second order Pareto model, when the second order coefficient is unknown. This problem is translated into a testing problem on the second order parameter. 
By constructing an appropriate model and an associated test statistic, we provide a uniform and adaptive confidence interval for the first order parameter. 
We also provide an almost matching lower bound, which proves that the result is minimax optimal up to a logarithmic factor.
\end{abstract}

\end{frontmatter}

\section{Introduction}
The Pareto model for the tail of a distribution is a useful tool to understand extremal phenomena. Indeed, the Fisher--Tippett--Gnedenko theorem states that a necessary and sufficient condition for the convergence in law of the (rescaled) maximum of i.i.d.~samples to a Fr$\acute{\text{e}}$chet distribution is that the tail of the distribution of the sample is regularly varying. We say that the distribution $F$ is regularly varying with parameter $\tau$ in the tail if the following holds:
\begin{equation}\label{paretolike}
1-F(x) = x^{-\tau}l(x),
\end{equation}
where $l(x)$ is slowly varying at infinity, that is, such that for any $x>0$, $\lim_{t \rightarrow \infty} l(tx)/l(t) = 1$.

There has been considerable work on estimating $\tau$, and many estimators (Hill's estimator of \cite{hill1975}, Pickands' estimator of \cite{pickands1975} for instance) have been suggested. Such estimators are consistent for the true parameter $\tau$ for distributions in the model~(\ref{paretolike}). In order to obtain a rate of convergence for these estimators, some additional assumptions have to be made. A typical one is the so-called (exact) Hall condition $1-F(x)-Cx^{-\tau} = C' x^{-\tau(\beta+1)}+o(x^{-\tau(\beta+1)}),$ where $\tau>0$ is the first order parameter, and $\beta>0$ is the second order parameter. The second order parameter characterizes the degree of proximity of the set of distributions to the exact Pareto distribution with parameters $\tau,C$. Here we assume a more general condition that we will call the second order Pareto assumption. That is, one considers the set of distributions $F$ that satisfies for some $C,C'$
\begin{equation}\label{SOP0}
|1-F(x)-Cx^{-\tau}| \leq C' x^{-\tau(\beta+1)}.
\end{equation}
This condition is a relaxation of the exact Hall condition in that the tail of the distribution does not have to satisfy the limiting form in the exact Hall model. We write $\set(\tau, \beta, C, C')=:\set(\tau, \beta)$ for the set of distributions satisfying (\ref{SOP0}). 

Assuming that the parameter $\beta$ is known, a classical result \citep[for instance, see][]{drees2001} under the exact Hall model states that Hill or Pickands' estimators using the information of $\beta$ satisfy (for $\epsilon>0$)
\begin{equation}\label{eq:est}
\sup_{F \in \bigcup_{\tau >\epsilon}\set(\tau, \beta)}| \hat \tau - \tau| =O_P(\gamma_n) =  O_P(n^{-\beta/(2\beta+1)}).
\end{equation}
Limiting distributions under the second order Pareto condition~(\ref{SOP0}) are obtained in~\cite{hall1982} and \cite{hall1984}. Also, a matching lower bound for this rate of estimation is proved in~\cite{hall1984} and \cite{drees2001}. Moreover, this type of bound (\ref{eq:est}) can be used in order to create a confidence interval for $\tau$ of width of order $\gamma_n$ with known $\beta$ \citep[for the asymptotic confidence interval under the exact Hall model, see][]{cheng2001, haeusler2007}. 

However, $\beta$ is unknown in general. There has been some work on estimating $\tau$ under various other assumptions in this case \citep[e.g. see][]{hall1985adaptive, drees1998, danielsson2001, carpentierkim}. In particular, under the second order Pareto model (\ref{SOP0}), \cite{carpentierkim} prove that it is possible to construct adaptive estimator of $\tau$ whose risk is the same (up to a $(\log\log(n))^{\frac{\beta}{2\beta+1}}$ factor) as the oracle rate $\gamma_n = n^{-\frac{\beta}{2\beta+1}}$ which is the optimal rate when $\beta$ is known.

The goal of this paper is to construct uniform and adaptive confidence intervals for $\tau$ when $\beta$ is unknown. 
That is, we want to build confidence intervals that have controlled coverage and optimal width $\gamma_n :=\gamma_n(\beta)$ (adaptively to the ``true" $\beta$ of the function) \textit{uniformly} over the set of second order Pareto distributions. This question is closely related to the problem of estimating~$\beta$. 
One possible approach for estimating $\beta$ is to restrict the set of distributions to those verifying a third order condition, which is more restrictive than both conditions (\ref{SOP0}) and the exact Hall condition. Under the third order condition, it is possible to estimate $\beta$ with a rate depeding on the third order parameter \citep[for more details, see][]{beirlant2008, gomes2008}, which might be used to construct the confidence interval. 
However, it is not clear which types of conditions for the parameter space are necessary for our goal. In fact, it has been an open question whether it is possible to construct uniform and adaptive confidence intervals for $\tau$ with unknown $\beta$ under the relaxed condition (\ref{SOP0}).
Using information theoretic bounds, this paper reveals the minimal conditions under which uniform and adaptive confidence intervals can be constructed. We discover that both the exact Hall condition and the third order condition are not required for this task by showing that the model we consider is strictly larger than those two cases (see Subsection \ref{subsec:cicontinuum} for more details). Although we do not provide tight constants, our result is minimax optimal (up to a logarithmic term). 
  

Related issues were considered in the domain of non-parametric functional estimation \citep[e.g.][]{low1997nonparametric, juditsky2003nonparametric, robins2006adaptive, gine2010confidence, hoffmann2011adaptive, bull2011adaptive, carpentier2013Lp}. In this area, one wants to construct a confidence set around a smooth  function with the smoothness parameter $s$. Similar to the case where estimating $\tau$ will depend on the unknown second order parameter $\beta$, the minimax rate of estimation over the set of $s-$smooth functions depends on the unknown smoothness $s$. 
Then, the oracle width of a confidence interval should depend on $s$. 
These papers investigate the case where $s$ is not available. In a first instance, they consider a simpler but related problem where one wants to decide between only two possible smoothness $s_0>s_1>0$. They state that it is neither possible to test between $s_0$ and $s_1$ uniformly over the set of smooth functions, nor to construct uniform and adaptive confidence sets on the whole model. However, it is uniformly possible on a restricted model where one removes some functions (a ring around the set of $s_0-$smooth functions, that is, functions that are $s_1$ smooth but close to $s_0$ smooth functions). These papers prove the minimax-optimal size of the set of functions that one has to remove.

In this paper, we construct an adaptive confidence interval for $\tau$ based on a testing procedure on the second order parameter, and show that this is minimax optimal up to a logarithmic factor.
We first consider the testing problem between $H_0: F \in \set(\tau, \beta_0)$ and $H_1: F \in \set(\tau, \beta_1)$ for $\beta_0>\beta_1>0$, and propose a test statistic for solving this problem (see Equation~(\ref{teststat})).
As we will prove in Section \ref{main},  
it is impossible to test between $H_0$ and $H_1$ uniformly over the whole set of $\beta_1$ second order Pareto distributions, and also impossible to construct an adaptive and uniform confidence interval for $\tau$. However, by removing a specific region of the set of second order Pareto distributions, we characterize a model that is maximal in a minimax sense, and for which the constructed confidence interval is uniform and adaptive over the class (\ref{model}) of distributions. We then use this testing idea developed in the two-points case for treating the case of a continuum of $\beta$. We provide, also in this case, a construction of an adaptive and uniform confidence interval for $\tau$. The model on which we prove that an uniform and adaptive confidence interval for $\tau$ exists is larger than the models considered in previous works such as~\cite{beirlant2008, gomes2008, gomes2012}. Moreover, we explain how to modify our method to consider a wider class of distributions such that the second order term is regularly varying in the tail. 
Finally we illustrate how to construct our adaptive intervals and compare several confidence intervals based on Hill's estimator with various sample fractions by simulations.

\section{Setting}\label{sec:setting}
Let $\mathcal D$ be the set of distributions on $\mathbb R^+$ that are c$\grave{\text{a}}$dl$\grave{\text{a}}$g (continuous on the right, limit on the left). We define the following subset of heavy tailed distributions (often referred to as second order Pareto distributions in the literature) given four parameters $\tau,\beta,C,C'>0$.
\begin{align*}\label{SOP}
\set(\tau, \beta, &C, C') \\
&:=  \left\{ F \in \mathcal D: \forall x \ \text{s.t.} \ F(x) \in (0,1], |1-F(x)-Cx^{-\tau}| \leq C' x^{-\tau(\beta+1)} \right\}. 
\end{align*}
From the definition, $F \in \set(\tau, \beta, C, C')$ satisfies a Pareto-like tail condition when $x \rightarrow~\infty$, that is, $1-F(x) \sim x^{-\tau}$ with the first order parameter $\tau$.
In addition, we note that $\beta$ characterizes the proximity of the distributions to the exact $\tau$-Pareto distribution.
Indeed, if $\beta$ is large, then the distribution is close to the $\tau$-Pareto distribution, while if $\beta$ is small, it is further away from the exact $\tau$-Pareto distribution. Note that in the particular case $\beta = \infty$, $\set(\tau,\beta,C,C')$ boils down to containing only one function $F_0(x) = 1-Cx^{-\tau}$ for $x \geq C^{1/\tau}$.

Let us first consider the simple case of distinguishing between two given second order parameters $\beta_0$ and $\beta_1$. 
We let $\beta_0 > \beta_1>0$, and consider two sets of second order Pareto distributions 
$\set(\tau,\beta_0,C,C')$ and
$\set(\tau,\beta_1,C,C')$.
Since  $\set(\tau,\beta_0,C,C') \subset \set(\tau,\beta_1,C, C')$, there does not exist a uniformly consistent test (see Definition \ref{unif_test}) for the following hypotheses
$$H_0: F \in \set(\tau,\beta_0, C, C') \quad \textbf{vs.} \quad H_1: F \in \set(\tau,\beta_1, C, C').$$
In order to get around this problem, we restrict the set $\set(\tau, \beta_1, C,C')$ to distributions which are \textit{not too close} to $\set(\tau,\beta_0,C,C')$.
Closeness is measured in the following sense,
\begin{equation*}\label{measure}
\|F - F_0\|_{\infty, \tau} := \sup_{x: F(x) \in (0,1)} \left|x^\tau (1-F(x)) - x^\tau(1- F_0(x)) \right|.
\end{equation*}
Then, the modified set of $\set(\tau,\beta_1, C,C')$ is defined as
\begin{equation}\label{modifiedset}
\tilde{\set}(\tau,\beta_1, \beta_0, C,C',\rho_n) := \left\{ F \in \set(\tau,\beta_1, C,C'): \|F -\set(\tau,\beta_0,C,C')\|_{\infty, \tau} \geq \rho_n \right\}.
\end{equation}
It is straightforward to check that $\tilde{\set}(\tau,\beta_1, \beta_0,C,C',0) = \set(\tau,\beta_1,C,C')$, but for $\rho_n>0$, these sets are proper subsets of $\set(\tau,\beta_1,C,C') \backslash \set(\tau,\beta_0,C,C')$. 

Consider now the modified testing problem
$$H_0': F \in \set(\tau,\beta_0, C, C') \quad \textbf{vs.} \quad H_1': F \in \tilde\set(\tau,\beta_1,\beta_0, C, C',\rho_n).$$
We recall that a statistical test $\Psi_n$ is a measurable function that takes values in $\{0,1\}$ of $n$ i.i.d. observations $X_1, \ldots, X_n$ from some distribution $F$. We say that a test is uniformly consistent if the rejection probability $\sup_{F \in H_0}\ee_{F} \Psi_n$ under the null hypothesis and non-rejection probability $\sup_{F \in H_1} \ee_{F} (1-\Psi_n)$ under the alternative hypothesis become small when $n$ increases. 
In other words, for a uniformly consistent test, both the type I error and type II error are uniformly well controlled by some predetermined level, which is denoted as $\alpha$.
More formally, we define an $\alpha$-uniformly consistent test in Definition \ref{unif_test}.
\begin{definition}[$\alpha$-uniformly consistent test]\label{unif_test}
A test $\Psi_n$  between two hypotheses $H_0$ and $H_1$ is $\alpha$-uniformly consistent if for $n$ large enough, we have
$$\sup_{F \in H_0}\mathbb E_F \Psi_n + \sup_{F \in H_1}\mathbb E_F (1-\Psi_n) \leq \alpha.$$
\end{definition}
In order for a test to verify this condition, we emphasize that both errors should be bounded in the worst case, that is, in a minimax sense, not only in a pointwise sense. This problem is related to uniformly consistent estimation of the parameter $\beta$ of a distribution $F$. Indeed, if one can construct many such tests on refined grids of an interval $[b,B]$, one can deduce from the outcome of these tests an estimate of $\beta$ that will be uniformly consistent for a model that will depend on $\rho_n$ (see Subsection \ref{subsec:cicontinuum}).

$\alpha$-uniformly consistent test can be useful for constructing an $\alpha$-adaptive and uniform confidence interval for $\tau$ over a model $\mathcal P_n$, and for a set $\mathcal I_b$ of parameters $\beta$. We let $\mathcal{I}_1$ be the possible range of $\tau$ and $\mathcal{I}_2$ be the possible range of $C$.
A confidence interval $C_n$ is a subset of $[0, \infty)$, and
 the diameter $|C_n|$ of $C_n$ is the length of the interval $C_n$.
We now provide the following definition of $\alpha$-adaptive and uniform confidence interval for $\tau$.
\begin{definition}[$\alpha$-adaptive and uniform confidence interval for $\tau$]\label{def:CI}

A confidence interval $C_n$ for $\tau$ is $\alpha$-adaptive and uniform for a model $\PP_n$, two sets $\mathcal I_1, \mathcal I_2$ and a constant $C'$ over a set $\mathcal I_{b}$, if there exists a constant $M$ such that for $n$ large enough, the following two conditions are satisfied simultaneously.
\begin{equation}\label{def2.2.1}
\sup_{\beta \in \mathcal I_{b}} \ \  \sup_{F \in  \bigcup_{\tau \in \mathcal I_1, C \in \mathcal I_2} \left\{ \set(\tau, \beta, C, C')\bigcap \PP_n \right\}} \pp_{F} (|C_n|  > M n^{-\frac{\beta}{2\beta+1}}) \leq \alpha, 
\end{equation}
and
\begin{equation}\label{def2.2.3}
\inf_{F \in \PP_n} \pp_F (\tau \in C_n) \geq 1-\alpha.
\end{equation}
\end{definition}
Inequality (\ref{def2.2.1}) implies that the diameter of the confidence interval is not larger than the oracle diameter, where the oracle diameter is the same as the minimax-optimal rate of estimation of the parameter $\tau$ under the correct model for $\beta$. In other words, it means that the confidence interval is adaptive to $\beta$. When $\beta$ can take only two values, for instance $\mathcal I_b = \{\beta_0, \beta_1\}$, such a confidence interval is adaptive over the two points $\beta_0$ and $\beta_1$. One could also consider a continuous range of $\beta$, for instance $\mathcal I_b = [b,B]$ where $0<b<B$, such that a confidence interval would be adaptive over a continuum of parameters. In both cases, the model $\mathcal P_n$ has to be restricted in order to ensure the existence of such a test, and is typically smaller than respectively $\bigcup_{\tau \in \mathcal I_1, C \in \mathcal I_2} \set(\tau, \beta_1, C, C')$ and $\bigcup_{\tau \in \mathcal I_1, C \in \mathcal I_2} \set(\tau, b, C, C')$. Inequality (\ref{def2.2.3}) implies that the confidence interval contains the true parameter with high probability. Again, this definition demands uniformity over the whole model $\PP_n$. In the two points case $\mathcal I_b = \{\beta_0, \beta_1\}$, we will consider $\mathcal P_n =  \bigcup_{\tau \in \mathcal I_1, C \in \mathcal I_2}  \Big(\set(\tau,\beta_0,C,C') \bigcup \tilde{\set}(\tau,\beta_1, \beta_0,C,C',\rho_n) \Big)$ where $\rho_n$ is specified later. In the more general case $\mathcal I_b = [b,B]$, the definition of $\mathcal P_n$ will be more involved and described later.


\section{Main results}\label{main}

We consider two settings. In a first instance, we assume preliminary knowledge of $\tau$ and $C$. This is rather a toy setting, but it is useful for us to understand precisely the mechanisms of the problem. In a second instance, we extend the ideas used for this simple setting to the case where $\tau$ and $C$ are unknown.

\subsection{$\alpha$-uniformly consistent test when $\tau,C$ are known}

We consider in this subsection only two possible values for the parameter $\beta$, that we write $\beta_0, \beta_1$ with $0<\beta_1<\beta_0$.

We first assume that $(\tau,C)$ are available, as well as $C'$ (which should be upper bounded). In this setting, the problem of building uniform and adaptive confidence intervals for $\tau$ is meaningless. But we can still cast in a meaningful way the problem of building a uniformly consistent test between
\begin{equation} \label{twosets}
\set_0(\tau,C):= \set(\tau,\beta_0,C,C') \quad \text{and} \quad \tilde \set_1(\tau, C, \rho_n) := \tilde \set(\tau,\beta_1,\beta_0, C, C' \rho_n).
\end{equation}
We consider the following testing problem
$$
H_0: F \in  \set_0(\tau,C) \quad \textbf{vs.} \quad H_1: F \in \tilde \set_1(\tau,C,\rho_n),
$$
and we are interested in the minimal order of $\rho_n$ such that there exists a uniformly consistent test for this problem.

\begin{theorem}\label{thm:test1}
Let $\alpha>0$, $\tau,C>0$, and $\beta_0>\beta_1>0$ (and it may be that $\beta_0 = \infty$) be given. Then,
\bed
\item[ \textbf{A.}] An $\alpha$-uniformly consistent test exists for some sequence $(\rho_n)_n$ such that $$\lim \sup_{n \rightarrow \infty} \rho_n n^{\beta_1/(2\beta_1+1)} < \infty.$$

\item[\textbf{B.}] If there exists an $\alpha$-uniformly consistent test, then necessarily
$$
\lim \inf_{n \rightarrow \infty} \rho_n n^{\beta_1/(2\beta_1+1)} >0.
$$
\eed
\end{theorem}

Theorem~\ref{thm:test1} states that $n^{-\beta_1/(2\beta_1+1)}$ is the minimax-optimal order which ensures the existence of an $\alpha$-uniformly consistent test between $H_0$ and $H_1$. Note that this $\rho_n$ is the same as the rate of estimation under the alternative hypothesis. 
The test statistic we propose in this setting is based on a simple idea: estimating $1-F(x)$ by $\hat p(x)$, and then testing whether there exists a point $x$ such that with a small constant $c$,
\begin{align*}
 |x^{\tau} \hat p_x - C| \geq C'x^{-\tau\beta_0} + c\rho_n. \color{black}
\end{align*}
In other words, we test if the empirical distribution belongs to a small enlargement of $\set_0(\tau,C)$, where this enlargement does not intersect with $\tilde \set_1(\tau,C, \rho_n)$.

\subsection{Uniform and adaptive confidence interval for $\tau$ over two points $\beta_0$, $\beta_1$}

In this subsection, we still consider only two possible values $\beta_0, \beta_1$ for the parameter $\beta$, and the associated sets $\set_0(\tau,C):= \set(\tau,\beta_0,C,C')$ and $\tilde \set_1(\tau, C, \rho_n) := \tilde \set(\tau,\beta_1,\beta_0, C, C' \rho_n)$.
But we extend the previous toy example results for the case where $\tau, C$ are not available (upper bound on the parameter $C'$ is available). Then we are interested in testing 
$$
H_0: F \in  \bigcup_{\tau \in \mathcal{I}_1, C \in \II_2}\set_0(\tau,C) \quad \textbf{vs.} \quad H_1: F \in  \bigcup_{\tau \in \II_1, C \in \II_2}\tilde \set_1(\tau,C, \rho_n),
$$
where $\II_1,\II_2$ are two closed intervals of $(0,\infty)$. 

A natural idea in this case is to plug estimators of $\tau$ and $C$ in the test statistics which are used in the case where $\tau$ and $C$ are known. Doing this leads to the following theorem.
\begin{theorem}\label{thm:test2}
Let $\alpha>0$ and $\beta_0>\beta_1>0$ (and it may be that $\beta_0 = \infty$) be given. Let $\tau \in \II_1, C \in \II_2$ be two unknown parameters.
An $\alpha$-uniformly consistent test exists for some sequence $(\rho_n)_n$ such that $$\lim \sup_{n \rightarrow \infty}  \rho_n \frac{n^{\beta_1/(2\beta_1+1)}}{\log(n)} < \infty.$$
\end{theorem}

We lose a $\log(n)$ factor with respect to the previous result. This comes from the fact that we have to estimate $\tau$ and $C$. 
Even though we do not know whether this factor is necessary or not, we know from Theorem~\ref{thm:test1} that it is not possible to deviate more than this $\log(n)$ factor (the lower bound of Theorem~\ref{thm:test1} applies a fortiori to this enlarged model).

The previous result is immediately translated into the existence of adaptive and uniform confidence intervals for $\tau$.
\begin{theorem}\label{thm:CI2}
Let $\alpha>0$ and $\beta_0>\beta_1>0$ (and it may be that $\beta_0 = \infty$) be given.  Let $\tau \in \II_1, C \in \II_2$ be two unknown parameters. Also with the notation (\ref{modifiedset}) and (\ref{twosets}), we set the model as follows.
\begin{equation}\label{model}
\PP_n = \Big(\bigcup_{\tau \in \II_1, C \in \II_2}\set_0(\tau,C)\Big) \bigcup \Big(\bigcup_{\tau \in \II_1, C \in \II_2}\tilde \set_1(\tau,C, \rho_n)\Big).
\end{equation}
Then,
\bed
\item[ \textbf{A.} ] An $\alpha$-adaptive and uniform confidence interval for $\tau$ in the model $\PP_n$ over $\mathcal I_b = \{\beta_0, \beta_1\}$ exists for some sequence $(\rho_n)_n$ such that $$\lim \sup_{n \rightarrow \infty} \rho_n \frac{n^{\beta_1/(2\beta_1+1)}}{\log(n)} < \infty.$$
\item[ \textbf{B.} ]  If there exists an $\alpha$-adaptive and uniform confidence interval for $\tau$ in the model $\PP_n$ over $\mathcal I_b = \{\beta_0, \beta_1\}$, then necessarily
$$
\lim \inf_{n \rightarrow \infty} \rho_n n^{\beta_1/(2\beta_1+1)} >0.
$$
\eed
\end{theorem}
In Theorem \ref{thm:CI2}, the first claim \textbf{\textit{A.}}  follows from Theorem~\ref{thm:test2}. Indeed, provided an $\alpha$-uniformly  consistent test between $H_0$ and $H_1$, one is able to choose adaptively the sample fraction for estimating~$\tau$ (see Section \ref{subsec:CI2}). Then the risk of this estimate depends on the $\beta$, that is,  the risk is of order $n^{-\frac{\beta}{2\beta+1}}$ where $\beta \in \{\beta_0, \beta_1\}$ is the ``true" parameter. On the other hand, \textbf{\textit{B.}} follows easily from the previously established lower bound.

{\textit{Remark:}} The upper bounds in both Theorems~\ref{thm:test2} and~\ref{thm:CI2} are not matching the lower bounds in Theorems~\ref{thm:test1} and~\ref{thm:CI2} by a $\log(n)$ factor. It is unclear to us whether this $\log(n)$ is necessary or not, but we conjecture that at least some power of $\log(n)$ is necessary, because of the uncertainty on $\tau,C$. Indeed, it is the use of estimators for $\tau$ and $C$ that causes this $\log(n)$ factor. We however believe that proving a matching lower bound with this additional $\log(n)$ factor is very involved, since one would need to consider a composite alternative \textit{and} a composite null hypothesis in the construction of the lower bound (if one of these two is not composite, it is actually possible to construct a test and a confidence interval without the additional $\log(n)$ factor in $\rho_n$). We leave this as an open problem for a future research.

\subsection{Uniform and adaptive confidence interval for $\tau$ over a continuum of parameter $\beta  \in [b,B]$}\label{subsec:cicontinuum}

In this subsection, we extend the results in the two classes to the case of a continuous set of parameters such that $\beta \in [b,B]$, where $0<b<B$. The key idea is to discretize the set $[b, B]$ into about $\log(n)$ number of disjoint intervals and do the successive testings. 
Let the number of grid points be $M_n :=\lfloor \log (n)/\xi \rfloor$ with a positive constant $\xi$. We first discretize the interval $[b,B]$ by a grid of points that are $1/{M_n}$-apart from each other. That is, we let $\beta_i = B - \frac{i(B-b)}{M_n}$ such that $B = \beta_0 > \beta_1 > \beta_2 > \ldots > \beta_{M_n} = b$.
Separation rate $\rho_n(\beta)$ is defined as a function of $\beta$, such that $\rho_n(\beta) = \max \left\{2(E(\alpha/(9M_n))\log(n)+D\log((9M_n)/\alpha), 2C' \right\} n^{-\frac{\beta}{2\beta+1}}$, similarly as we defined in (\ref{rhondefinition}) in the two points test.
We also extend the definition of the modified set $\tilde \set$ in (\ref{modifiedset}) by introducing $\tilde \set(\tau,\beta_{i+2}, \beta_i, C,C', \rho_n(\beta_{i+2}))$ to be the set of $\beta_{i+2}$-second order Pareto distributions that is $\rho_n(\beta_{i+2})$ away from $\set(\tau,\beta_i,C,C')$,
\begin{align}
\tilde{\set}&(\tau,\beta_{i+2}, \beta_i, C,C', \rho_n(\beta_{i+2})) \nonumber\\
 &:= \left\{ F \in \set(\tau,\beta_{i+2}, C,C'): \|F -\set(\tau, \beta_i,C,C')\|_{\infty, \tau} \geq \rho_n(\beta_{i+2}) \right\}.\label{generalrhon}
\end{align}
Consider the model
\begin{equation}\label{newmodel} 
\mathcal P_n =\bigcup_{ \tau \in \mathcal{I}_1, C \in \mathcal{I}_2}  \left( \set(\tau, B,C,C') \bigcup \Biggl( \bigcup_{i=0}^{M_n-2}   \tilde \set(\tau,\beta_{i+2}, \beta_i,C,C', \rho_n(\beta_{i+2}))\Biggr)  \right),
\end{equation}
where $\mathcal I_1, \mathcal I_2$ be two closed sets of $(0, \infty)$.

Let us define, for a given distribution $F\in \set(\tau, b,C,C')$, the index $\beta^*:=\beta^*(F)$ as
\begin{equation}\label{betastar}
\beta^*(F) = \sup_{\beta \in [b,B]} \{\beta: F \in \set(\tau, \beta, C,C')\}.
\end{equation}
This supremum is well defined since it is upper bounded by $B$, and the set $\{\beta: F \in \set(\tau, \beta, C,C')\}$ is non-empty since it contains $b$. This $\beta^*(F)$ can be thought of as the ``intrinsic" $\beta$ of $F$, i.e.~the one that characterizes the complexity of the model to which $F$ belongs.

\begin{theorem}\label{generaltesting}
 Let $\alpha>0$, and let $\mathcal P_n$ be the model defined in (\ref{newmodel}). For a positive constant $\xi$, there exists a sequence $(\rho_n(\beta_i))_{n,i}$ such that
\begin{equation}\label{rhonconti}
\lim \sup_{n \rightarrow \infty} \sup_{0\leq i\leq \lfloor \log (n)/\xi\rfloor }\rho_n(\beta_i) \frac{n^{\beta_i/(2\beta_i+1)}}{\log(n)\log\log(n)^{3/2}} < \infty,
\end{equation}
such that there exists an estimate $\hat \beta$ of $\beta^*(F)$ which satisfies 
\begin{equation*}
\lim \sup_{n \rightarrow \infty} \sup_{F \in \mathcal P_n} \mathbb P_F \Big(|\hat \beta - \beta^*(F)| \geq 4\xi\frac{(B-b)}{\log(n)}\Big) \leq \alpha.
\end{equation*}
\end{theorem}

Theorem \ref{generaltesting} states that on the model  $\mathcal P_n$ defined in (\ref{newmodel}), it is possible to estimate $\beta^*$ with the rate $1/\log(n)$. On the one hand, this rate seems very slow; but it is the price to pay for considering a rather wide model. On the other hand, as we will see in the next Theorem, it is actually sufficient in order to obtain an adaptive and uniform confidence set. The idea in the proof of Theorem~\ref{generaltesting} is somewhat similar to the successive testing procedure considered in~\citet[Section 2.5]{hoffmann2011adaptive}. 

\begin{theorem}\label{generaltesting2}
Let $\alpha>0$. There exists a sequence $(\rho_n(\beta_i))_{n,i}$ satisfying (\ref{rhonconti})
and such that on the model $\mathcal P_n$ defined in (\ref{newmodel}) over $\mathcal I_b = [b,B]$, there exists an $\alpha$-adaptive and uniform confidence interval $C_n$ for $\tau$.

\end{theorem}


The model~\eqref{newmodel} is not very easy to interpret as such, but it is actually a rather general model. In particular, consider a class $\mathcal{G}$ of distributions defined for $0<C_1<C'$ and $D>0$ as
\begin{align}
\mathcal G &:= \bigcup_{\beta \in \mathcal [b ,B - 2(B-b)/M_n], \tau \in \mathcal I_1, C \in \mathcal I_2}\Big\{F \in \mathcal S(\tau,\beta,C,C'): \nonumber\\ 
&C_1 x^{-\tau(\beta+1)} \leq |1-F(x) - C x^{-\tau}| \leq C' x^{-\tau(\beta+1)} \hspace{2mm}\mathrm{for}\hspace{2mm}x>D  \Big\}.\label{eq:self}
\end{align}
Then it is possible to prove that $\mathcal{G}$ is included in the set (\ref{newmodel}) when we pick $\xi$ large enough for a sufficiently large $n$ as shown below.
\begin{lemma}\label{lem:self}
Let $(\rho_n(\beta_i))_{n,i}$ be a sequence satisfying (\ref{rhonconti}). Then if $\xi$ is chosen such that $ \big( C_1 - 2C'e^{-\xi (B-b)/(2(2B+1))}\big) > 0$, then for $n$ large enough, we have
$$\mathcal G \subset \mathcal P_n.$$
\end{lemma}
For the constant in the grid points, large $\xi$ means that we use coarse grids so that we remove smaller regions of the parameter space. That is, large $\xi$ corresponds to the large model so that it yields a wide confidence interval for $\tau$. Nevertheless, this large choice of $\xi$ will soften the requirements on the choice of $C_1$ and $C'$ such that $\mathcal{G}$ is contained in the model (\ref{newmodel}).
This set $\mathcal{G}$ is actually larger than the set of distributions verifying the third order condition in the literature \citep[e.g.][]{gomes2012}, i.e.~the set of distributions that verify for $\gamma>0$
\begin{equation}\label{comp:model1}
|1-F(x) - C x^{-\tau} - C_1 x^{-\tau(\beta+1)} | \leq C' x^{-\tau(\gamma + \beta+1)},
\end{equation}
since convergence in the tail to a distribution of the form $1-F(x) - C x^{-\tau} - C_1 x^{-\tau(\beta+1)}$ is not imposed. The set $\mathcal{G}$ is also larger, when $n$ grows to $\infty$, than the set of distributions that satisfy the exact Hall condition,
\begin{equation}\label{exachall}
 1-F(x) = Cx^{-\tau} + \tilde C x^{-\tau(\beta+1)} + o(x^{-\tau(\beta+1)} ).
\end{equation}
This implies in particular that our model $\mathcal P_n$ is much larger than usual models where adaptive estimators or adaptive and uniform confidence intervals are derived. In fact, $\mathcal{G}$ can be understood as an analogue to the set of self-similar functions \citep[see Condition 3 of][]{gine2010confidence}.

\section{Discussion}\label{sec:discussion}

\subsection{Construction of the test statistic and the confidence interval}

 The test statistics (\ref{teststat}) involves the empirical distribution of data and estimators of $\tau$ and $C$. For instance, in the two classes testing, we can use Hill's estimate and the associated estimate of $C$ with the sample fraction corresponding to the null (if $\Psi_n = 0$) or alternative hypotheses (if $\Psi_n = 1$). Another option is to consider an adaptive estimate of $\tau$ and construct an adaptive confidence interval centred on this estimate. We illustrate the practical construction of the estimate we propose in Algorithm~\ref{fig:sampling2}, in the case of a continuum of parameters $\beta$ (Theorem~\ref{generaltesting2}).


\begin{algorithm}[t]
\caption{Practical construction of the confidence interval for $\tau$}\label{fig:sampling2}

\begin{algorithmic}
\STATE {\bf Parameters:} For given $B, b(\geq 0.5), C', \alpha, \xi$, we let $\beta_i = B - \frac{i(B-b)}{M_n}$ where $M_n :=\lfloor \log (n)/\xi \rfloor$, and
$$\rho_n(\beta_i) = (\log \log (n))n^{-\beta_i / (2\beta_i+1)}.$$
\vspace{-0.5cm}
\FOR{$i:=M_n:2$}
\STATE {\bf Estimates:} Estimate $\tau,C$ by $\hat \tau_i, \hat C_i$ with the sample fraction associated to $\beta_i$ (for instance, for $\hat \tau_i$, we can use the inverse of Hill's estimator computed with the $\lfloor n^{2\beta_i/(2\beta_i+1)} \rfloor$ largest samples, or the adaptive estimate~\citep{carpentierkim}; for $\hat C_i$, we can use $n^{1/(2\beta_i+1)} \sum_{i=1}^n \one_{\{X_i > n^{1/(\hat \tau_i(2\beta_i+1))}\}}/n$).  
\STATE {\bf Set} $\hat B_i = n^{1/\vartheta_i}$ where $\vartheta_i = \Big(\hat \tau_i + (\log \log(n))n^{-\frac{\beta_i}{2\beta_i+1}}\tilde c_1(\alpha,i)\Big)(2\beta_i+1)$ with $\tilde c_1(\alpha,i) = q_{1-\alpha/2} \hat \tau_i$ where $q_{1-\alpha/2}$ is the $\alpha/2$ quantile of a $\mathcal N(0,1)$ \citep[following][combined with a delta method]{cheng2001,haeusler2007}
\STATE  {\bf Set} $\hat p_x := \frac{1}{n} \sum_{i=1}^n \one_{\{ X_i > x \}}$
\STATE  {\bf Set} $T_n(i) = \sup_{x \leq \hat B_i} \Big(|x^{\hat \tau_i} \hat p_x - \hat C_i| - C' x^{-\hat \tau_i\beta_{i-2}}\Big)$.
\STATE  {\bf Set} $\Psi_n(i) = \one\{ T_n(i) \geq \rho_n(\beta_i)/2\}$.
\IF{$\Psi_n(i) = 1$}
\STATE  {\bf Set} $d(C_n) = \tilde c_1(\alpha,i) (\log \log(n)) n^{-\beta_i/(2\beta_i+1)}$.
\STATE  {\bf Set} $\tilde \tau = \hat \tau_i$ 
\STATE  {\bf Set} $C_n = [\tilde \tau - d(C_n), \tilde \tau + d(C_n)]$.
\ENDIF
\ENDFOR
\STATE  {\bf Return} $C_n$.
\vspace{2pt}
\end{algorithmic}
\end{algorithm}

The choice of $B$ and $b$  depends on the belief of the user about the magnitude of the parameter $\beta$, i.e.~of the models one wishes to consider. The parameter $\xi$ corresponds to the desired precision on $\beta$. The choice of the constant $\tilde c_1(\alpha,i)$ is in practice made following methods \citep[]{cheng2001, haeusler2007} combined with a delta-method, and is thus chosen lower than the bound derived in the proofs of this paper, which is conservative. Finally, $\alpha$ corresponds to the desired coverage of the confidence set. These parameters can be fixed arbitrarily according to the user's preference. However, there is no simple answer for how to choose the constant $C'$. We recall that the larger $C'$ is, the larger the class of $\beta-$second order Pareto distributions becomes. This implies that the theory would be more complete if we could handle the class of $\beta-$second order Pareto distributions with the second order condition for any $C'<\infty$. But then, a larger $C'$ would yield a larger separation zone $\rho_n$ as well as the necessity of considering wider confidence intervals for $\tau$. We believe that the choice of the upper bound for $C'$ should depend on the specific problem considered, and also on whether one is interested in asymptotic rates (in which case $C'$ has to slowly go to infinity with $n$), or in the final width of the confidence set for not-too-small tail probabilities (in which case $C'$ has to be of reasonably small magnitude). A reasonable heuristic for fixing $C'$ is as follows. For each candidate $\beta_i$, we fix
$$C':= \sup_{1\leq x \leq \hat B_i} |x^{\hat \tau_i} \hat p_x - \hat C_i|(1+ c_n)+0.2,$$
where $c_n$ is a confidence bound as e.g.~$c_n = \sqrt{\log(1/\alpha)} n^{-\beta_i/(2\beta_i+1)}$ and $0.2$ makes $C'$ not-too-small. This heuristic is efficient, in particular if the model error is maximized for small $x$, which is often the case in practice, and in usual parametric heavy tailed models (Fr\'echet, Student, etc).

The choices of $\beta_i$, $\hat C_i$ from the threshold $\lfloor n^{2\beta_i/(2\beta_i+1)}\rfloor$ follow from the theory, and with these choices, the results are minimax-optimal up to a  $\log(n)\big(\log\log(n)\big)^{3/2}$ factor (see Theorem \ref{thm:CI2} and \ref{generaltesting}). Ideally, to obtain better constants in the bounds,  one would like to tune the constants in all these quantities. In order to do this, however, it is necessary to estimate the model bias (i.e.~the deviance with respect to Pareto distribution) precisely. This is possible in more restrictive models than ours \citep[e.g.][]{gomes2012}, since they assume the third order condition for the model bias (definite shape order plus a negligible term). It is thus possible to estimate this model bias with few parameters from a finite sample size using the information regarding the bias whose shape is guaranteed to hold on the entire tail. 

In contrast, we consider a setting which is much broader than this one, and where the model bias is only upper bounded, which implies that the bias in our setting does not have any definite shape. It is thus not possible to infer from a finite sample the shape of the (far right) tail of the model bias, which makes it difficult to tune the constants without having an oracle knowledge of the distribution. 
Although we agree on the practical impact for more refined tuning, we believe that it is not possible in this broader model to tune exactly without a priori knowledge of the problem; this is the price to pay for a broader model, which is particularly relevant for discrete distributions.

\subsection{Distributions that are in $\mathcal P_n$ but not in the models~\eqref{comp:model1} and~\eqref{exachall}}
\label{subsec:largermodel}


We have shown in Section \ref{subsec:cicontinuum}  that our model $\mathcal{P}_n$ (c.f. Equation~\eqref{newmodel}) is strictly larger than the usual models such as \eqref{comp:model1} and~\eqref{exachall} as well as the set of self-similar distributions (see Equation~\eqref{eq:self}). Then it is of interest to see if there are useful models in applications that belong to our model, but that are out of the more restrictive classes previously considered. To that end, we consider the class of heavy tailed distributions that take only a countable set of values, i.e.~discrete distributions. There are many examples of such discrete heavy tailed distributions; either for ``natural" reasons (e.g.~the distributions of wages in a population) or for rounding issues (e.g. hydrology measures).
Thus for practical applications, it is very important to consider these distributions. 

%
We claim here that many discrete distributions that are in the domain of attraction of a Fr\'echet distribution are \textit{not} in the models~\eqref{comp:model1} and~\eqref{exachall}, while some of them are in our model $\mathcal P_n$. Let us write $\mathcal H$ for the distributions that satisfy Equation~\eqref{exachall} (and who thus also satisfy~\eqref{comp:model1}). The following lemma states that distributions which are discretized (with equally spaced grids) versions of distributions in $\mathcal{H}$ are not in $\mathcal{H}$. 
\begin{lemma}\label{lem:discdist}
Let $F \in \mathcal H$ with parameters $\tau, \beta>0$. Let $a_0 \geq 1, t>0$, and set $a_i = a_0 + it$ for any $i \in \mathbb N$. Let $\tilde F$ be the discretized version of $F$ according to $(a_j)_j$, i.e. for all $x\geq 1$, $\tilde F(x) = F(a_i)$ where $a_i$ is the largest element of $\{a_j, j\in \mathbb N\}$ smaller than or equal to $x$.

If $\beta>1/\tau$, then 
$$\tilde F\not \in \mathcal H.$$
On the other hand, if $\min(\beta, 1/\tau) \in (b,B)$, then for $n$ large enough,
$$\tilde F\in \mathcal P_n,$$
where $\mathcal P_n$ is defined in~\eqref{newmodel}. 
\end{lemma}
The proof of this lemma is in Subsection~\ref{proof:selfi}.

In order to illustrate this lemma, we consider the distribution
\begin{equation}\label{disc_paretomodel}
\tilde F: x \in [1, \infty] \rightarrow 1-\lfloor x\rfloor^{-\tau},
\end{equation}
which puts mass only on integers, and is equal to a $\tau-$Pareto distribution on these integers. Clearly, this distribution is such that $1 - \tilde F(x) = x^{-\tau} + o(x^{-\tau})$ so it is in the domain of attraction of a Fr\'echet distribution.
Lemma~\ref{lem:discdist} implies that this distribution does not belong to $\mathcal H$, i.e.~to the models~\eqref{comp:model1} and~\eqref{exachall}, for any constants $\beta, C, C_1, C', \tilde C$. This reveals that even simple heavy tailed discrete distributions are not in the usually considered models~\eqref{comp:model1} and~\eqref{exachall}. On the other hand, this distribution is in the model $\mathcal P_n$ defined in Equation~\eqref{newmodel} by Lemma~\ref{lem:discdist}. Lemma~\ref{lem:discdist} holds for many other discrete distributions as well as this simple discrete distribution.

\subsection{Extension to the case of a regularly varying second-order term} 
Condition~\eqref{paretolike} is a necessary and sufficient condition for a distribution $F$ to be in the domain of attraction of a Fr\'echet distribution. However, as mentioned in the introduction, this model is too broad to allow for uniform rates of convergence for an estimate of $\tau$.

Setting~\eqref{SOP0} was introduced to characterize a set of parameters in which it is possible to estimate $\tau$ at a uniform rate. This rate depends on the second order parameter $\beta$ that characterizes the maximal distance between a distribution included in (\ref{SOP0}) and the exact Pareto distribution.

As an alternative to the setting (\ref{SOP0}), we might consider the set of distributions $F$ which satisfy 
\begin{equation}\label{slowlyvarying}
1-F(x) = Cx^{-\tau} (1 + R(x)),
\end{equation}
where $|R|$ is upper bounded by a $\tau \beta$-regularly varying function at $+\infty$ \citep[see Lemma 2.3.2 of][]{dekkers1993optimal}. 

From the form of (\ref{slowlyvarying}), we know that this alternative setting is an extension of the \textit{exact} Hall condition (\ref{exachall}), i.e.~$R(x) = \tilde C x^{-\tau\beta} + o(x^{-\tau\beta} )$. 
In addition, by definition of regularly varying functions, if a function verifies Equation (\ref{slowlyvarying}) (where $R$ is $\tau \beta$-regularly varying), then it verifies 
%
\begin{equation}\label{cond:sv}
|R(x)| \leq D x^{-\tau \beta + \phi(x)},
\end{equation}
where $D$ is a constant, and $\phi$ is a function such that $\lim_{x \rightarrow \infty}\phi(x) = 0$. This condition is slightly weaker than Equation (\ref{slowlyvarying}), but these two conditions are very much related.

The condition (\ref{SOP0}) we consider in this paper for the two points test and confidence interval is weaker than the exact Hall condition~\eqref{exachall} (although it is stronger than (\ref{cond:sv}) since it requires $\phi(x) = 0$). Indeed, condition~\eqref{exachall} implies
$$\lim_{x \rightarrow \infty} \frac{1-F(x) - Cx^{-\tau}}{ x^{-\tau(1+\beta)}} = \tilde C,$$
while our condition~\eqref{SOP0} does not impose the existence of this limit; it gives just an upper bound on the distance to the exact Pareto distribution. Moreover, we have proved that up to a $\log(n)$ factor, the model $\mathcal P_n$ that we derive from the condition (\ref{SOP0}) is the largest possible in a minimax sense for studying the question of uniform and adaptive confidence intervals. In the case of a continuum of parameters $\beta$, the model (\ref{newmodel}) that we define for constructing adaptive and uniform confidence intervals contains the set of distributions that verify the exact Hall condition. It is larger (see Lemma~\ref{lem:self}) than the set \eqref{eq:self}, which clearly includes the set of distributions in the exact Hall model. 
We emphasize that most papers considering estimation of $\beta$ actually consider a more restrictive setting than the exact Hall condition~\eqref{exachall}, as e.g.~the \textit{third order condition} \citep[see][for instance]{gomes2012}.


Before considering our method in the more general setting (\ref{cond:sv}), we want to make an important remark. In the setting (\ref{cond:sv}), $\beta$ does not characterize the uniform rates of convergence for the estimator of $\tau$ anymore \citep[see][]{drees1998optimal,drees2001}. Indeed, if $\phi$ converges to $0$ too slowly, the distance between $F$ and the exact Pareto distribution can become too large, and the rate $n^{-\frac{\beta}{2\beta+1}}$ for estimating $\tau$ is then out of reach. The main reason why we are interested in testing $\beta$ in the model~\eqref{SOP0} is because in this model, $\beta$ characterizes the complexity of the model to which $F$ belongs. More specifically, $\beta$ controls the bias to an exact Pareto distribution, and yields the estimation error being $n^{-\frac{\beta}{2\beta+1}}$. 
Also, through this complexity, knowing $\beta$ should allow us to (a) compute the optimal sample fraction for obtaining optimal estimators of $\tau$ and (b) have an idea of the precision of these estimators so that we can build a confidence interval for $\tau$. However, under condition~\eqref{cond:sv}, the role of $\beta$ is less clear, in particular if $\phi$ decays very slowly to $0$---actually, as soon as $\lim_{x \rightarrow \infty} |\log(x)\phi(x)| = \infty$. In this case, the minimax-optimal rate of estimation for $\tau$ and the optimal width of the confidence interval is \textit{not} $n^{-\frac{\beta}{2\beta+1}}$, but it is larger \citep[see][]{drees1998optimal, drees2001}. Thus, testing whether $\beta=\beta_0$ or $\beta=\beta_1$, or estimating $\beta$ under the general model (\ref{cond:sv}) does not provide an answer for confidence statements.

To closer look at the construction of uniform confidence interval in this alternative setting (\ref{cond:sv}), let us consider the wider model 
$$|R(x)| \leq x^{-\tau \beta + \phi(x)}.$$
In the extreme case $|R(x)| = x^{-\tau \beta + \phi(x)}$ and if $\lim_{x \rightarrow \infty} |\log(x)\phi(x)|=\infty$, as mentioned above, $\beta$ is not driving the rate of convergence of an optimal estimate of $\tau$ and the length of the associated confidence interval. The quantity that is driving the optimal rate can depend on $n$ and is defined as follows. Let $n>0$ be the number of samples. Let $x_n$ be the point such that it equalizes the bias and the standard deviation of an estimate \citep[as the one in][]{carpentierkim} computed only with the samples larger than this point, i.e.~$x_n$ be such that
$$\sqrt{\frac{1}{(1-F(x_n))n}} = |R(x_n)|.$$
Let $\beta_n$ be such that $n^{-\frac{\beta_n}{2\beta_n+1}}$ is equal to the risk of this estimate (or equivalently, let $\beta_n$ be such that $x_n = n^{\frac{1}{\tau(2\beta_n+1)}}$). Then we define $\beta_n^* = \inf_{N\geq n} \beta_N$. This $\beta_n^*$ is actually the quantity of importance (and not $\beta$): it is the quantity that characterizes the performance of the estimate for a fixed $n$. And although $\beta_n^* \rightarrow \beta$ as $n\rightarrow \infty$, the rate of convergence of this quantity can be arbitrarily slow depending on how slowly $\phi(x)$ goes to $0$---when this convergence is too slow and $|\log(x)\phi(x)| \rightarrow \infty$, the rate is \textit{not} $n^{-\frac{\beta}{2\beta+1}}$. Hence, in this setting and for a fixed $n$, one does not want to test $\beta$ (or do confidence intervals according to $\beta$), but we want to test the quantity $\beta_n^*$ (and build confidence intervals according to $\beta_n^*$). Actually, the estimate defined in~\cite{carpentierkim} computed with all the points larger than $n^{\frac{1}{\tau(2\beta_n^*+1)}}$ attains the optimal risk $n^{-\frac{\beta_n^*}{2\beta_n^*+1}}$ \citep[as the one defined in][]{drees1998optimal}. And if one does not know $\beta$ (and also $\beta_n^*$), by modifying slightly the proof of Theorem 3.6 in the paper~\citep{carpentierkim}, one can prove that the adaptive estimate $\hat \tau$ from the paper~\citep{carpentierkim} satisfies that with high probability,
$$|\hat \tau - \tau| = O\Big(\big(\frac{n}{\log\log(n)}\big)^{-\frac{\beta_n^*}{2\beta_n^*+1}}\Big).$$

There is then an easy extension of our procedure to test whether  $\beta_n^*$ equals $\beta_0$ or $\beta_1$. The only change with respect to the procedure we proposed is in the definition of the test statistic $T_n$ introduced in Section 5.1, which should be defined as
$$\tilde T_n = \sup_{x: [n^{\frac{1}{\tau(2\beta_0 + 1)}}, n^{\frac{1}{\tau(2\beta_1 + 1)}}]} \Big(|x^{ \tau} \hat p_x - C| - C' x^{-\tau\beta_0}\Big),$$
i.e.~the distance to the model is considered only for large enough $x$. This test will be consistent if the alternative hypothesis is restricted to distributions which are at least $n^{-\frac{\beta_1}{2\beta_1 + 1}}$ away from distributions such that $\beta_n^* = \beta_0$. The idea behind this change is to test deviations from the exact Pareto distribution only for large $x$ so that $R(x)$ is small enough under the null hypothesis (under the condition~\eqref{SOP0}, $R(x)$ is always properly bounded for any $x$). To extend the case where $\tau$ is unknown, the test statistic should be replaced by an estimate similar to (\ref{teststat}) as in the proof of Theorems~\ref{thm:test2}.

\section{Numerical experiments}\label{sec:numerical}

\subsection{Experiments on continuous distributions}

We consider usual parametric extreme value distributions:
\begin{itemize}
\item $\tau$-Pareto distributions on $[1, \infty)$ with $\tau \in \{1,2\}$. We remind that the $\tau$-Pareto distribution $F_{\tau}^P(x) =1- x^{-\tau}$  is included in $\mathcal{S}(\tau, \infty, 1,0)$.
\item $\tau$-Fr$\acute{\text{e}}$chet distributions on $[0, \infty)$ with $\tau \in \{1,2\}$. The $\tau-$Fr$\acute{\text{e}}$chet distribution $F_{\tau}^F(x) =\exp(-x^{-\tau})$ is included in $\mathcal{S}(\tau, 1, C,C')$ for some constants $C,C'$.
\item $\tau$-Student distributions with $\tau$ degrees of freedom on $\rr$ with $\tau \in \{1,2\}$. 
This distribution is in $\mathcal{S}(\tau, 1/(2\tau), C,C')$ for some constants $C,C'$. \item Cauchy distribution on $\mathbb{R}$. The standard  Cauchy distribution $f(x) = \frac{1}{\pi} \frac{1}{1+x^2}$ is included in $\mathcal{S}(1, 2, C,C')$ for some constants $C,C'$.
\end{itemize}
The last two distributions are symmetric at 0 and defined on $\mathbb R$, so we consider the absolute value of the samples so that Hill's estimate with large sample fraction still exists ($\tau$ and $\beta$ do not change).

We compute confidence sets around $1/\tau$ (not around $\tau$), since most empirical studies in the literature focus on $1/\tau$, which enables us to compare more easily to these results. We follow the algorithm in Algorithm~\ref{fig:sampling2} for computing $1/\hat \tau_i$ for each $\beta_i$ (using Hill's estimate with the $\lfloor n^{\frac{2\beta_i}{2\beta_i+1}}\rfloor$ largest samples). For the estimation of $1/\tau$, the corresponding $\tilde c_1(\alpha,i)$ is $q_{1-\alpha/2}\widehat{1/ \tau_i}$, but not $q_{1-\alpha/2}\hat \tau_i$ \citep[see][]{cheng2001, haeusler2007}. We fix $C'$ according to the heuristic discussed in Section~\ref{sec:discussion}, and consider $[b,B] = [0.5,10]$, $\xi = \log(n)/95$ and $\alpha = 0.05$. We denote our confidence interval as $AdapCI$. 

The other methods we choose to compare are derived according to \cite{haeusler2007} and \cite{cheng2001}.
 We first estimate $\beta$ by $\hat \beta$ as discussed in \citet[$\S$3]{cheng2001}. We then use it to compute the number of the largest samples $\hat k$ that we will consider. We will consider three different values $\hat k^*, \tilde k, k_{CP}$ of $\hat k$. We first consider the optimal estimated sample number $\hat k^*:= \lfloor n^{2\hat \beta/(2\hat \beta+1)}\rfloor$. This sample fraction should provide the best results in terms of estimating $1/\tau$. Second, we use a sample number that is a $o(\hat k^*)$, namely $\tilde k:=\lfloor \hat k^*/\sqrt{\log n} \rfloor$. The rational behind this heuristic is that the asymptotic normality of Hill estimate with known constants (i.e.~with negligible bias) holds if and only if the sample number is $o(n^{2\beta/(2\beta+1)})$, and in this case exact asymptotic coverage of the confidence interval can be achieved. Third, we use $\hat k_{CP}$ suggested in \citet[$\S$3]{cheng2001}. The idea behind this heuristic is to provide a coverage that is as close to the theoretic coverage $1-\alpha$ as possible.

  The confidence interval will then be computed using the $\hat k$ largest samples, according to two methods discussed in~\cite{haeusler2007}: the \textit{Wald} and \textit{score} type confidence intervals. A first step is to estimate $1/\tau$ by Hill's estimate $\widehat{1/\tau (\hat k)}$ with the sample number $\hat k$. The confidence intervals will then be centred on this estimate. The \textit{Wald} type confidence interval is obtained by \citet[p.177]{haeusler2007}
$$
\Big( (1-\hat k^{-1/2}q_{1-\alpha/2})\widehat{1/\tau (\hat k)}, (1+\hat k^{-1/2}q_{1-\alpha/2})\widehat{1/\tau (\hat k)})\Big),
$$
where
\begin{align*}
\widehat{1/\tau (\hat k)} &:= \frac{1}{\hat k} \sum_{i=1}^{\hat k} \log X_{(n-i+1)} - \log X_{(n-\hat k)} \\
X_{(n-k)} &\leq X_{(n-k+1)} \leq \ldots \leq X_{(n)}.
\end{align*}
The \textit{score} type confidence interval is obtained by
$$
\Big( (1+\hat k^{-1/2}q_{1-\alpha/2})^{-1} \widehat{1/\tau (\hat k)}, (1-\hat k^{-1/2}q_{1-\alpha/2})^{-1}\widehat{1/\tau (\hat k)}\Big).
$$
We denote these confidence intervals $W_{\hat k^*}, W_{\tilde k},W_{CP}$ for \textit{Wald}, and similarly $S_{\hat k^*}, S_{\tilde k},S_{CP}$ for \textit{score} method.

We iterate these procedures 100 times, and compute the number of times that obtained confidence intervals contain the true $1/\tau$ (\textit{coverage}) and the average of length of intervals (\textit{size}).
In order to compare these 7 methods, both coverage and size have to be taken into account, and the ranking of the methods we analyse is not straightforward. A good confidence interval should have high coverage and small size, but there is a trade-off between these two quantities.
 Although our focus is to provide confidence intervals, for the readers' interests, we also provide their mean values and MSEs from 100 iterations (see Table \ref{Pareto_sim2}, \ref{Student_sim2}, \ref{Frechet_sim2},  and \ref{Cauchy_sim2}).
\color{black}

A natural competitor would also be the method presented in \citet[Table 5]{gomes2012} since it is the only method (to the best of our knowledge) that is proven to be adaptive to the second order parameter---the results of \cite{haeusler2007} and \cite{cheng2001} are proven for a fixed sample fraction, and some fixed oracle sample fractions are discussed. They are however not proven for an adaptive sample fraction, and only heuristics with adaptive sample fractions are proposed in these papers. However, \cite{gomes2012} describe their method as a ``terribly time-consuming algorithm", and they display computational results \textit{with size and coverage of the confidence set} only for a Student distribution of parameter 2, for $n \in \{100, 200, 1000\}$ \citep[see Table 5 in][]{gomes2012}. Moreover, we found that our method, as well as the \textit{score} and \textit{Wald} methods, give much better results in this case, simultaneously in terms of size and coverage (see Table \ref{student_sim}). Thus we do not implement the method of \cite{gomes2012} on our experiments as a competitor, but we can still compare the results for the Student distribution of parameter 2.



We provide the simulation results on the coverage and size of the confidence intervals in Tables~\ref{pareto_sim} (Pareto), \ref{student_sim} (Student), \ref{frechet_sim} (Fr\'echet) and \ref{cauchy_sim} (Cauchy). We can see that our adaptive method $AdapCI$ gives fairly stable and small confidence intervals in terms of both the coverage and the size, and is particularly efficient for small sample sizes. The \textit{Wald} method provides also good results, both in terms of size and coverage, in particular with $\hat k^*$ number of samples (which provides the smallest confidence interval, with almost always good coverage). Our method is in most cases comparable to the associated method $W_{\hat k^*}$. In contrast, the \textit{score} method gives often a too wide confidence interval for small sample sizes $n=100,200$. For the case $\tau=2$ in Table \ref{student_sim}, we can compare our result with \citet[Table 5]{gomes2012}. The results in \citet[Table 5]{gomes2012} are almost always worse both in terms of coverage and size than the results of Table \ref{student_sim}. 


\begin{table}[ht!] \centering 
\scriptsize
\caption{Coverage probabilities (\textit{first row}), sizes (\textit{second row}) of the confidence intervals for $1/\tau$ for underlying $\tau$-Pareto distributions} \label{pareto_sim}
\ra{1.2}
\begin{tabular}{@{}l rrrr c rrrr@{}} 
  \toprule[1.2pt]
&\multicolumn{4}{c}{$\tau = 1$} && \multicolumn{4}{c}{$\tau = 2$}\\ 
\cmidrule{2-5} \cmidrule{7-10}
$n$ & 100 & 200 & 1000 & $10^4$ & \phantom{abc} & 100 & 200  & 1000 & $10^4$\\ 
\midrule
$AdapCI$ & 100 &100 & 97 & 85 && 99 & 100 & 99 &96\\
           & 0.66  & 0.52& 0.34 &  0.27 && 0.34 & 0.27 & 0.20& 0.16 \\
 \hline
$W_{\hat k^*}$ & 93 & 93 & 93 & 94 && 90 & 93 & 96 & 92 \\
           & 0.86  & 0.64 & 0.36 & 0.16 && 0.90 & 0.69 & 0.24& 0.06 \\
 \hline
$W_{\tilde k}$ &93 &92 & 94& 94&& 86 & 91 & 91& 96 \\
           & 1.30  & 1.01 & 0.59 & 0.28 && 0.90 & 0.87 & 0.39 & 0.11\\
\hline 
$W_{CP}$ &91 & 91 & 93 & 90 && 85 & 87 & 91 & 98 \\
            & 1.18  & 0.95 & 0.63 & 0.36  && 0.97 & 0.92 & 0.56 & 0.18 \\
 \hline
$S_{\hat k^*}$ &96 & 96 & 96 & 95 && 98 & 97 &93 & 94 \\
          & 1.14  & 0.72 & 0.38 & 0.16 && 4.45 & 3.66  & 0.26 & 0.06\\
 \hline
$S_{\tilde k}$ &96 & 95 &96 & 96 && 98 & 94 & 93 & 97 \\
           & 5.15  & 1.39 & 0.65& 0.29 && 7.21 & 6.18 & 0.50& 0.11\\
\hline
$S_{CP}$&97 & 92 & 97 & 92 && 97 & 94 & 93 & 99\\
           & 3.33  & 1.25 & 0.70 & 0.37 && 8.72 & 7.99 & 2.16 & 0.19\\
 \bottomrule[1.2pt]
\end{tabular} 
\end{table}

\begin{table}[ht!] \centering 
\scriptsize
\caption{Coverage probabilities (\textit{first row}), sizes (\textit{second row}) of the confidence intervals for $1/\tau$ for underlying $\tau$-Student distributions} \label{student_sim}
\ra{1.2}
\begin{tabular}{@{}l rrrr c rrrr@{}} 
  \toprule[1.2pt]
&\multicolumn{4}{c}{$\tau = 1$} && \multicolumn{4}{c}{$\tau = 2$}\\ 
\cmidrule{2-5} \cmidrule{7-10}
$n$ & 100 & 200 & 1000 & $10^4$ & \phantom{abc} & 100 & 200  & 1000 & $10^4$\\ 
\midrule
$AdapCI$ &80 & 80 & 55 & 93 && 100 & 100 & 98 & 96 \\
           & 0.55 & 0.48 & 0.39 & 0.33 && 0.41 & 0.33 & 0.33&  0.24 \\
 \hline
$W_{\hat k^*}$ & 56 & 67 & 55 & 47 && 83 & 86 & 74 & 70\\
           & 0.55 & 0.46 & 0.26 & 0.12 && 0.36 & 0.28 & 0.13 & 0.18 \\
 \hline
$W_{\tilde k}$ & 79 & 82 & 85 & 91 && 84 & 86 &  66 & 84 \\
            &  0.87 & 0.76 & 0.46 & 0.22 && 0.47 & 0.42 & 0.20 & 0.31 \\
 \hline
$W_{CP}$ & 70 & 79 & 85 &  96 && 82 & 84 & 80 & 87 \\
           &  0.71 & 0.67 & 0.46 & 0.27 && 0.57 & 0.50 & 0.26 & 0.41 \\
 \hline
$S_{\hat k^*}$ &  81 & 81 & 63 &  47 && 91 & 95 & 80 & 73 \\
           &   0.65 & 0.50 &  0.27 & 0.12 && 0.55 & 0.60 & 0.14 & 0.31 \\
 \hline
$S_{\tilde k}$ &  94  & 94 & 95 & 94 && 95 & 95 & 73 &  91\\
           &   1.26 &  0.92 & 0.48 & 0.22&& 1.45 & 1.84 & 0.29 &0.95 \\
 \hline
$S_{CP}$ &  89 & 90 & 93 & 97 && 93 & 96 & 88 & 91\\
           &  1.10 & 0.78 & 0.49 & 0.27 && 2.91&  1.84 & 0.37 & 1.11\\ 
 \bottomrule[1.2pt]
\end{tabular} 
\end{table}

\begin{table}[ht!] \centering 
\scriptsize
\caption{Coverage probabilities (\textit{first row}), sizes (\textit{second row}) of the confidence intervals for $1/\tau$ for underlying $\tau$-Fr$\acute{\text{e}}$chet distributions} \label{frechet_sim}
\ra{1.2}
\begin{tabular}{@{}l rrrr c rrrr@{}} 
  \toprule[1.2pt]
&\multicolumn{4}{c}{$\tau = 1$} && \multicolumn{4}{c}{$\tau = 2$}\\ 
\cmidrule{2-5} \cmidrule{7-10}
$n$ & 100 & 200 & 1000 & $10^4$ & \phantom{abc} & 100 & 200  & 1000 & $10^4$\\ 
\midrule
$AdapCI$ & 98 & 93& 85 & 96 && 81 & 81  & 84 & 95\\
           & 0.62  & 0.51& 0.40& 0.37 &&  0.34 & 0.38& 0.44& 0.32\\
 \hline
$W_{\hat k^*}$  & 83 & 89& 95& 88 && 50& 35& 47& 86\\
            & 0.70  & 0.56 & 0.33& 0.15 && 0.35 & 0.29& 0.32& 0.28\\
 \hline
$W_{\tilde k}$  &92 & 87& 98& 98 && 75& 60& 67& 93\\
           & 1.09  & 0.89& 0.55& 0.27 && 0.50 & 0.43& 0.46& 0.51\\
 \hline
$W_{CP}$ &88 & 88& 95& 96&& 79& 69& 76& 91\\
           & 0.99  & 0.86& 0.60& 0.36 && 0.63 & 0.51& 0.56& 0.66\\
 \hline
$S_{\hat k^*}$ &97 & 95& 96& 90 && 75& 49& 56& 92\\
           & 0.83  & 0.62& 0.34& 0.15 && 1.46 & 1.18& 1.85& 0.94\\
 \hline
$S_{\tilde k}$ &95 & 95& 99& 99 && 96& 81& 80& 97\\
           & 0.67  & 1.13& 0.59& 0.28 && 1.60 & 2.04& 3.20& 2.62\\

 \hline
$S_{CP}$ &96 & 95& 98& 93 && 97& 90& 91& 95\\
           & 1.39  & 1.07& 0.66& 0.38 && 3.62 & 2.71& 3.72& 2.82\\

\bottomrule[1.2pt]
\end{tabular}
\end{table}

\begin{table}[ht!] \centering 
\scriptsize
\caption{Coverage probabilities (\textit{first row}), sizes (\textit{second row}) of the confidence intervals for $1/\tau \equiv 1$ for underlying Cauchy distributions} \label{cauchy_sim}
\ra{1.2}
\begin{tabular}{@{}l rrrr@{}} 
  \toprule[1.2pt]
&\multicolumn{4}{c}{$\tau = 1$}\\ 
\cmidrule{2-5} 
$n$ & 100 & 200 & 1000 & $10^4$ \\ 
\midrule
$AdapCI$  &78 & 80& 52 & 94 \\
           & 0.57  & 0.46& 0.36& 0.37 \\
 \hline
$W_{\hat k^*}$  & 63 & 67& 61& 48\\
           & 0.58  & 0.48 & 0.26& 0.12 \\
 \hline
$W_{\tilde k}$ &81 & 89& 89& 89\\
          &0.94  & 0.80& 0.45& 0.22 \\
 \hline
$W_{CP}$  &76 & 82& 89& 91\\
            & 0.78  & 0.69& 0.46& 0.27 \\
 \hline
$S_{\hat k^*}$ &80 & 76& 69& 52\\
           & 0.77  & 0.52& 0.27& 0.12 \\
 \hline
$S_{\tilde k}$ &95 & 93& 94& 91\\
           & 1.44  & 0.98& 0.37& 0.22 \\
  \hline
$S_{CP}$  &88 & 90& 93& 94\\
           & 1.11  & 0.81& 0.48& 0.27\\
\bottomrule[1.2pt]
\end{tabular}
\end{table}

\begin{table}[ht!] \centering 
\scriptsize
\caption{Means (\textit{first row}), MSEs (\textit{second row}) of the estimates of $1/\tau$ for $\tau$-Pareto distributions} \label{Pareto_sim2}
\ra{1.2}
\begin{tabular}{@{}l rrrr c rrrr@{}} 
  \toprule[1.2pt]
&\multicolumn{4}{c}{$\tau = 1$} && \multicolumn{4}{c}{$\tau = 2$}\\ 
\cmidrule{2-5} \cmidrule{7-10}
$n$ & 100 & 200 & 1000 & $10^4$ & \phantom{abc} & 100 & 200  & 1000 & $10^4$\\ 
\midrule
$\widehat{1/\tau (\hat k)}$ & 0.995 & 0.998& 0.977 & 0.970  && 0.496 & 0.499  & 0.489 & 0.487\\
       & 0.013  & 0.007& 0.009& 0.005 &&  0.003 & 0.002& 0.002& $9.4 \ .1^{4}$\\
 \hline
$\widehat{1/\tau (\hat k^*)}$  & 0.936 & 0.987& 0.989& 0.997 && 0.492& 0.512& 0.512& 0.502\\
            & 0.051  & 0.030 & 0.009& 0.001 && 0.044 & 0.024& 0.004& $2.6 \ .1^{4}$\\
 \hline
$\widehat{1/\tau (\tilde k)}$  &0.929 & 0.960& 1.000& 0.996 && 0.486& 0.503& 0.504& 0.501\\
           & 0.103  & 0.055& 0.027 & 0.004 && 0.065 & 0.059& 0.011& $6.8 \  .1^{4}$\\
 \hline
$\widehat{1/\tau (k_{CP})}$ &0.923 & 0.966& 1.001& 0.989&& 0.479& 0.497& 0.492& 0.500\\
           & 0.087  & 0.046& 0.029 & 0.008 && 0.060& 0.067& 0.019& 0.002\\
\bottomrule[1.2pt]
\end{tabular}
\end{table}

\begin{table}[ht!] \centering 
\scriptsize
\caption{Means (\textit{first row}), MSEs (\textit{second row}) of the estimates of $1/\tau$ for  $\tau$-Student distributions} \label{Student_sim2}
\ra{1.2}
\begin{tabular}{@{}l rrrr c rrrr@{}} 
  \toprule[1.2pt]
&\multicolumn{4}{c}{$\tau = 1$} && \multicolumn{4}{c}{$\tau = 2$}\\ 
\cmidrule{2-5} \cmidrule{7-10}
$n$ & 100 & 200 & 1000 & $10^4$ & \phantom{abc} & 100 & 200  & 1000 & $10^4$\\ 
\midrule
$\widehat{1/\tau (\hat k)}$ & 0.860 & 0.837& 0.872 & 0.927  && 0.584 & 0.565  & 0.463 & 0.465\\
       & 0.030  & 0.032& 0.021& 0.007 &&  0.010 & 0.006& 0.007& 0.008\\
 \hline
$\widehat{1/\tau (\hat k^*)}$  & 0.855 & 0.840& 0.897& 0.934 && 0.467& 0.460& 0.471& 0.466\\
            & 0.054  & 0.044 & 0.017& 0.005 && 0.019 & 0.009& 0.002& 0.004\\
 \hline
$\widehat{1/\tau (\tilde k)}$  &0.923 & 0.891& 0.955& 0.965 && 0.455& 0.429& 0.452& 0.480\\
           & 0.085  & 0.056& 0.016 & 0.004 && 0.041 & 0.020& 0.007& 0.007\\
 \hline
$\widehat{1/\tau (k_{CP})}$ &0.890 & 0.889& 0.956& 0.976&& 0.456& 0.441& 0.463& 0.480\\
           & 0.076  & 0.051& 0.016  & 0.005 && 0.053& 0.024& 0.009& 0.013\\
\bottomrule[1.2pt]
\end{tabular}
\end{table}

\begin{table}[ht!] \centering 
\scriptsize
\caption{Means (\textit{first row}), MSEs (\textit{second row}) of the estimates of $1/\tau$ for  $\tau$-Fr$\acute{\text{e}}$chet distributions} \label{Frechet_sim2}
\ra{1.2}
\begin{tabular}{@{}l rrrr c rrrr@{}} 
  \toprule[1.2pt]
&\multicolumn{4}{c}{$\tau = 1$} && \multicolumn{4}{c}{$\tau = 2$}\\ 
\cmidrule{2-5} \cmidrule{7-10}
$n$ & 100 & 200 & 1000 & $10^4$ & \phantom{abc} & 100 & 200  & 1000 & $10^4$\\ 
\midrule
$\widehat{1/\tau (\hat k)}$ & 0.903 & 0.903& 0.941 & 0.954  && 0.424 & 0.414  & 0.440 & 0.461\\
       & 0.018  & 0.016& 0.009& 0.007 &&  0.009 & 0.014& 0.007& 0.007\\
 \hline
$\widehat{1/\tau (\hat k^*)}$  & 0.930 & 0.926& 0.961& 0.982 && 0.408& 0.416& 0.415& 0.465\\
            & 0.037  & 0.027 & 0.008& 0.002 && 0.019 & 0.022& 0.015& 0.011\\
 \hline
$\widehat{1/\tau (\tilde k)}$  &0.939 & 0.951& 0.997& 0.986 && 0.434& 0.418& 0.433& 0.475\\
           & 0.076  & 0.046& 0.015 & 0.005 && 0.029 & 0.027& 0.021& 0.026\\
 \hline
$\widehat{1/\tau (k_{CP})}$ &0.939 & 0.954& 0.998& 0.983&& 0.455& 0.424& 0.436& 0.505\\
           & 0.066  & 0.047& 0.020 & 0.009 && 0.036& 0.031& 0.031& 0.038\\
\bottomrule[1.2pt]
\end{tabular}
\end{table}

\begin{table}[ht!] \centering 
\scriptsize
\caption{Means (\textit{first row}), MSEs (\textit{second row}) of the estimates of $1/\tau$ for  $\tau$-Cauchy distributions} \label{Cauchy_sim2}
\ra{1.2}
\begin{tabular}{@{}l rrrr@{}} 
  \toprule[1.2pt]
&\multicolumn{4}{c}{$\tau = 1$}\\ 
\cmidrule{2-5} 
$n$ & 100 & 200 & 1000 & $10^4$ \\ 
\midrule
$\widehat{1/\tau (\hat k)}$  &0.846 & 0.848& 0.874 & 0.935 \\
           & 0.033  & 0.029 & 0.021& 0.007 \\
 \hline
$\widehat{1/\tau (\hat k^*)}$  &0.847 & 0.876& 0.895& 0.940\\
          &0.058 & 0.037 & 0.017& 0.004 \\
 \hline
$\widehat{1/\tau (\tilde k)}$  &0.904 & 0.944& 0.969& 0.977\\
            & 0.087 & 0.060& 0.020& 0.004 \\
 \hline
$\widehat{1/\tau (k_{CP})}$ &0.870 & 0.938 & 0.969& 0.986\\
           & 0.073  & 0.055& 0.020& 0.005 \\
\bottomrule[1.2pt]
\end{tabular}
\end{table}

\subsection{Experiments on discretized distributions}
As we claim in Subsection \ref{subsec:largermodel}, our model contains discretized Pareto distributions which are not contained in usual models previously considered. In this Subsection, we thus consider the model (\ref{disc_paretomodel}) $F:x \in [1, \infty) \rightarrow 1-\lfloor x\rfloor^{-\tau}$. As discussed in Subsection \ref{subsec:largermodel}, $F \in \mathcal S(\tau, 1/\tau, 1, C')$ for $C'$ large enough depending on $\tau$. We perform the experiments for the seven methods discussed in the last subsection, and for sample sizes $n \in \{100, 200, 1000, 10^4, 10^5\}$, and for $\tau \in \{1,2\}$.

Table \ref{discPareto_sim} shows the results for these class discretized distributions. All methods perform correctly in the case $\tau = 1$. However, for $\tau = 2$, the coverage probability of all the \textit{Wald} and \textit{score} methods are very small for a large sample size $n \in \{10^4, 10^5\}$. This comes from the fact that these methods over-estimate $\beta$ (which is $1/\tau = 1/2$ in this case), which implies the size of the confidence interval is too small for guaranteeing a good coverage. This problem is more acute for $\tau = 2$ that for $\tau = 1$ since $\beta = 1/\tau$ is smaller for $\tau = 2$. Our method $AdapCI$, on the other hand, detects that the complexity of this model is higher than in the case of the exact Pareto distribution, and increases the size of the confidence intervals, which guarantees a good coverage for the resulting confidence interval.


\begin{table}[ht!] \centering 
\scriptsize
\caption{Coverage probabilities (\textit{first row}), sizes (\textit{second row}) of the confidence intervals for $1/\tau$ for underlying discretized $\tau$-Pareto distributions} \label{discPareto_sim}
\ra{1.2}
\begin{tabular}{@{}l rrrrr c rrrrr@{}} 
  \toprule[1.2pt]
&\multicolumn{5}{c}{$\tau = 1$} && \multicolumn{5}{c}{$\tau = 2$}\\ 
\cmidrule{2-6} \cmidrule{8-12}
$n$ & 100 & 200 & 1000 & $10^4$ & $10^5$ & \phantom{abc} & 100 & 200  & 1000 & $10^4$ &$10^5$\\ 
\midrule
$AdapCI$ & 97 & 99& 97 & 96 & 99 && 60 & 96  & 96 & 99 & 100\\
           & 0.68  & 0.65& 0.62& 0.43 & 0.39 &&  0.47 & 0.54& 0.47& 0.35& 0.27\\
 \hline
$W_{\hat k^*}$ & 91 & 91& 92& 87& 94 && 45& 63& 86& 16& 7\\
           & 0.94  & 0.70 & 0.62& 0.29 & 0.07 && 0.37 & 0.46& 0.62& 0.07& 0.01\\
 \hline
$W_{\tilde k}$  &90 & 91& 87& 70& 94 &&75& 75& 87& 55& 12\\
           & 1.38  & 1.02& 1.01& 0.44 & 0.13 && 0.59 & 0.68& 0.83& 0.12& 0.03\\
 \hline
$W_{CP}$ &92 & 93& 89& 93& 99 && 71& 76& 90& 67& 30\\
           & 1.30  & 0.94& 1.04& 0.51 & 0.20 && 0.54 & 0.57& 0.75& 0.13& 0.03\\
  \hline
$S_{\hat k^*}$  &89 & 93& 94& 87& 94 &&51& 60& 86& 16& 7\\
           & 2.10  & 0.84& 0.70& 0.33 & 0.07 && 0.92 & 1.25& 3.51& 0.07& 0.01\\
 \hline
$S_{\tilde k}$  &96 & 94& 94& 96& 93 &&51& 72& 90& 50& 12\\
           & 4.74  & 2.35& 1.81& 1.44 & 0.13 &&2.46 & 2.15& 4.97& 0.12& 0.03\\
\hline
$S_{CP}$ &95 & 95& 96& 95& 96 && 59& 73& 91& 62& 31\\
           & 3.00  & 1.84& 2.12& 2.58 & 0.21 &&1.25 & 1.76& 4.11& 0.14& 0.03\\
\bottomrule[1.2pt]
\end{tabular}
\end{table}

\begin{table}[ht!] \centering 
\scriptsize
\caption{Means (\textit{first row}), MSEs (\textit{second row}) of the estimates of $1/\tau$ for discretized $\tau$-Pareto distributions} \label{discPareto_sim2}
\ra{1.2}
\begin{tabular}{@{}l rrrrr c rrrrr@{}} 
  \toprule[1.2pt]
&\multicolumn{5}{c}{$\tau = 1$} && \multicolumn{5}{c}{$\tau = 2$}\\ 
\cmidrule{2-6} \cmidrule{8-12}
$n$ & 100 & 200 & 1000 & $10^4$ & $10^5$ & \phantom{} & 100 & 200  & 1000 & $10^4$ &$10^5$\\ 
\midrule
\scriptsize{$\widehat{1/\tau (\hat k)}$}  & \scriptsize{1.006} & \scriptsize{1.020}& \scriptsize{0.999} & \scriptsize{0.906} & \scriptsize{0.981} && \scriptsize{0.371} & 0.444  & 0.467 & 0.421& 0.416\\
           & \scriptsize{0.022}  & 0.015& 0.015& 0.056 & 0.002 &&  0.027 & 0.011& 0.008& 0.019 & 0.019\\
 \hline
\scriptsize{$\widehat{1/\tau (\hat k^*)}$} & \scriptsize{1.061} & 1.035& 1.020& 1.031& 1.007 && 0.511& 0.588& 0.530& 0.556& 0.456\\
           & \scriptsize{0.068} & 0.033 & 0.025& 0.010 & \tiny{$4.4 \ .1^{5}$} && 0.040 & 0.051& 0.028& 0.017& 0.006\\
 \hline
\scriptsize{$\widehat{1/\tau (\tilde k)}$}  &\scriptsize{1.064} & 1.025& 1.026& 1.016& 0.999 && 0.630& 0.592& 0.578& 0.528& 0.493\\
           & \scriptsize{0.121}  & 0.070& 0.076& 0.022 & \tiny{$8.5 \ .1^{4}$}  && 0.081 & 0.065 & 0.059& 0.005& 0.008\\
 \hline
\scriptsize$\widehat{1/\tau (k_{CP})}$ &\scriptsize{1.074} & 1.039& 1.038& 1.031& 0.994 && 0.631& 0.612& 0.551& 0.524& 0.512\\
           & \scriptsize{0.127}  & 0.059& 0.085& 0.035 & 0.002 && 0.071 & 0.056& 0.037& 0.004& 0.002\\
\bottomrule[1.2pt]
\end{tabular}
\end{table}

\section{Technical proofs}
\subsection{Proof of the upper bound in Theorem~\ref{thm:test1} (Proof of [A.])}\label{proof:theorem1A}




In this Subsection, we write for simplicity $\set_0 = \set_0(\tau,C)$, $\set_1 = \set_1(\tau,C)$ and $\tilde \set_1 = \tilde \set_1(\tau,C, \rho_n)$.

Let $X_1, \ldots, X_n$ be an i.i.d. random sample from a distribution $F \in \mathcal{S}_1$. We write, for any $x \in \rr^+$,
\begin{equation*} 
p_x := \mathbb P(X>x) = 1 - F(x),
\end{equation*}
and its empirical estimate that we define for all $x$ rationals (we write $\mathbb Q$ for the rationals) larger than $0$
\begin{equation*}
\hat p_x := \frac{1}{n} \sum_{i=1}^n \one_{\{ X_i > x \}}.
\end{equation*}
For the $x$ that are non rational, we set $\hat p_x = \lim_{y \in \mathbb Q, y>x, y \rightarrow x} \hat p_y$. 

We propose the following test statistic
\begin{align*}
T_n = \sup_{x \leq n^{\frac{1}{\tau(2\beta_1+1)}}} \Big(|x^{\tau} \hat p_x - C| - C' x^{-\tau\beta_0}\Big).
\end{align*}
The test is of the form
\begin{align*}
\Psi_n = \one_{\{ T_n \geq \rho_n/2\}},
\end{align*}
where $\rho_n \geq \max(4D\log(1/\alpha),2C') n^{-\frac{\beta_1}{2\beta_1+1}}$ with a universal constant $D$ in Lemma \ref{lem:lardev}.
Then, we reject the null if $\Psi_n = 1$, and vice versa.

The following results in Lemma \ref{lem:doublehypo} show that the test statistics $T_n$ is a reasonable criterion for this testing problem. Lemma \ref{lem:lardev} proves that the difference between empirical estimate and the true probability is controlled uniformly well.

\begin{lemma}\label{lem:doublehypo}
Assume that $\rho_n \geq 2C' n^{-\frac{\beta_1}{2\beta_1+1}}$. Then
\begin{itemize}
\item[(i)] $F \in H_0$ implies $\sup_x \Big(|x^{\tau} p_x - C| - C' x^{-\tau\beta_0}\Big) \leq 0$.
\item[(ii)] $F \in H_1$ implies that $\sup_{x \leq n^{\frac{1}{\tau(2\beta_1+1)}}} \Big(|x^{\tau} p_x - C| - C' x^{- \tau\beta_0}\Big) \geq \rho_n$.
\end{itemize}
\end{lemma}

\begin{lemma}\label{lem:lardev}
Suppose we have an iid sample from $F \in \set_1$. With probability larger than $1-\alpha$, we have
\begin{align*}
\sup_{x \leq n^{\frac{1}{\tau(2\beta_1+1)}}} x^{\tau} | \hat p_x - p_x | &\leq  D n^{-\frac{\beta_1}{2\beta_1+1}} \sqrt{\log\big(\frac{1}{\alpha}\big)} +n^{-\frac{2\beta_1}{2\beta_1+1}} \log \big(\frac{1}{\alpha}\big) \\
&\leq  2D n^{-\frac{\beta_1}{2\beta_1+1}} \log\big(\frac{1}{\alpha}\big),
\end{align*}
where $D$ is some constant that depends only on $\beta_1$, on a lower bound on $\tau$ and on an upper bound on $C,C'$.
\end{lemma}

Now, we combine the results obtained in Lemma \ref{lem:doublehypo} and \ref{lem:lardev} by considering two hypotheses separately.
Let $\alpha >0$.

\begin{description}
\item[Under $H_0$:] We obtain that with probability larger than $1-\alpha$
\begin{equation*}
\sup_{x \leq n^{\frac{1}{\tau(2\beta_1+1)}}} \Big(|  x^{\tau}\hat p_x - C| - C'x^{-\tau\beta_0}\Big) \leq  2D n^{-\frac{\beta_1}{2\beta_1+1}} \log(1/\alpha) < \rho_n/2.
\end{equation*}
\item[Under $H_1$:] We obtain that with probability larger than $1-\alpha$
\begin{equation*}
\sup_{x \leq n^{\frac{1}{\tau(2\beta_1+1)}}} \Big(|  x^{\tau}\hat p_x - C| - C'x^{-\tau\beta_0}\Big) \geq  \rho_n - 2D n^{-\frac{\beta_1}{2\beta_1+1}} \log(1/\alpha) > \rho_n/2.
\end{equation*}
\end{description}
This concludes the proof of the upper bound in Theorem \ref{thm:test1}.
\begin{proof}[Proof of Lemma \ref{lem:doublehypo}]
(i) is clear by definition of $H_0$, but proof of (ii) is more involved.

First, we define three regions $R = \{x:  | 1-F(x) - Cx^{-\tau}| \leq C' x^{-\tau (1+\beta_0)}\}$, $R^+ = \{ x: 1-F(x) \geq Cx^{-\tau}+ C'x^{-\tau(1+ \beta_0)}\}$, and $R^- = \{ x: 1-F(x) \leq Cx^{-\tau}- C'x^{-\tau(1+ \beta_0)}\}$. Then for any $F \in H_1$, we define $\tilde F$ as follows,
\begin{align*}
\tilde F(x) = F(x)\one_{\{x \in R\}} &+ (1-Cx^{-\tau}-C' x^{-\tau (1+\beta_0)}) \one_{\{x\in R^+\}}\\
 &+ (1-Cx^{-\tau}+C'x^{-\tau (1+\beta_0)})
\one_{\{x \in R^-\}}.
\end{align*}
We have $\tilde F \in \set_0$ by definition. 

By definition of $\tilde F$, we have
\begin{align}
\big|x^{\tau} &(\tilde F(x) -F(x))\big| = \left| x^{\tau} [(1-F(x)) - (Cx^{-\tau}+C'x^{-\tau(1+\beta_0)})] \right|\one_{\{x \in R^+\}}  \nonumber \\
&+  \left| x^{\tau} [(1-F(x)) - (Cx^{-\tau}-C'x^{-\tau(1+\beta_0)})] \right|\one_{\{x \in R^-\}} \nonumber \\
&= \left( x^{\tau} p_x - C-C'x^{-\tau\beta_0} \right) \one_{ \{x \in R^+\}} + \left( C- x^\tau p_x - C'x^{-\tau\beta_0}\right) \one_{\{x\in R^-\}}. \label{approx}
\end{align}

But then, by the fact that $F \in \set_1$, by using the upper bound for $x^{\tau}p_x \leq C +C' x^{-\tau\beta_1}$ and the lower bound for $x^\tau p_x \geq C-C' x^{-\tau \beta_1}$, (\ref{approx}) can be upper bounded as follows,
$$
(\ref{approx}) \leq 2(C'x^{-\tau \beta_1} - C'x^{-\tau \beta_0}) \leq 2C' x^{-\tau \beta_1}.
$$
Recall that by definition of $\tilde \set (\beta_1, \rho_n)$, for any $F_0 \in \set_0$, there exists $x_0$ such that
\begin{align*}
|x_0^{\tau} [(1-F(x_0)) - (1-F_0(x_0))]| \geq \rho_n.
\end{align*}
Thus, we can restrict the set of $x_0$ such that
$$
\rho_n \leq 2C' x_0^{-\tau \beta_1} \Leftrightarrow x_0 \leq \left( \frac{2C'}{\rho_n}\right)^{-1/(\tau \beta_1)} \leq n^{1/(\tau(2\beta_1+1)}.
$$

Since $\tilde F\in \set_0$, this implies that there exists $x_0 \leq  n^{1/(\tau(2\beta_1+1)}$ such that
\begin{align*}
|x_0^{\tau} [(1-F(x_0)) - (1-\tilde F(x_0))]| \geq \rho_n.
\end{align*}
Combining this with Equation~(\ref{approx}), it is clear that the maximum point should be either in $R^+$ or $R^-$.
Either way, we have found that there exists $x_0$ such that 
$$
|x_0^\tau p_{x_0} - C| - C'x_0^{-\tau \beta_0} \geq \rho_n.
$$
This concludes the proof of Lemma \ref{lem:doublehypo}.
\end{proof}

\begin{proof}[Proof of Lemma \ref{lem:lardev}]
For notational convenience, let $B := n^{1/(\tau(2\beta_1+1))}$.

First, we split the interval $[1, B]$ into $K :=\lfloor \log B \rfloor$ number of disjoint intervals such that 
$I_0 := [e^0,e^1), I_1 = [e^1, e^2), \ldots, I_k = [e^{k}, e^{k+1}), \ldots, I_{K-1} = [e^{K-1}, e^{K})$.

The proof is based on the Talagrand's inequality \citep{talagrand} of the form in \citet[Theorem 7.3]{bousquet2003concentration} and \citet[Equation (13)]{bull2011adaptive} after using the fact that $(1+t)\log(1+t)-t \geq \min (t^2/3, t/2)$ for $t \geq 0$.  See also~\citep{wellnerbook, boucheron2013concentration}.
\begin{theorem}[Talagrand's inequality]
Let $\FF$ be a countable set of functions from $\Omega$ to $\rr$ and assume that all functions $f$ in $\FF$ are measurable and takes values in $[-1/2,1/2]$. 
Denote
$$
Z = \sup_{f \in \FF} \left| \sum_{i=1}^n (f(X_i) - \ee f(X)) \right|.
$$
Let $\sigma \leq 1/2$, and $V$ be any two numbers satisfying
$$
\sigma^2 \geq \sup_{f \in \FF} \ee f^2, \ \ V \geq n\sigma^2 + 2 \ee Z.
$$
Then for all $t \geq 0$, 
$$
\pp \left( Z\geq \ee(Z) + t \right) \leq \exp\left(- \min \left[\frac{t^2}{3(n\sigma^2+2\ee(Z))}, \frac{t}{2}\right] \right).
$$
\end{theorem}

Here we set $f(X_i) = f_x(X_i) =\frac{1}{2} \one \{ X_i > x\}$, and $\FF = \{ f_x,  x \in I_k \bigcap \mathbb Q \}$. For this class of functions, we have $\frac{1}{n}\sum_{i=1}^n f_x(X_i) = \hat p_x/2$ and $\ee f_x(X_i) = p_x/2$. Also, we can let $\sigma^2 = \frac{1}{4} \sup_{x \in I_k}  [p_x] \leq \min(\frac{(C+C')}{4} e^{-k\tau}, 1/4)$.

The following result is obtained by applying the last theorem to the class $\FF$, and by rescaling. Denoting $\mu_k :=  \ee\sup_{x \in I_k \bigcap \mathbb Q} |\hat p_x - p_x | $, we have
$$
\pp\left( \sup_{x \in I_k\bigcap \mathbb Q} | \hat p_x - p_x | >  \mu_k+t \right) \leq \exp\left( - \frac{1}{12} \left( \min \big(\frac{nt^2}{(\sigma^2+2\mu_k)}, nt\big)\right)\right).
$$

Since the function $F$ is cadlag, and since the rationals are dense in the real line, by construction of $\hat p_x$, the last inequality implies the following corollary.
\begin{corollary}\label{cor:wvv}
Denoting $\mu_k :=  \ee\sup_{x \in I_k \bigcap \mathbb Q} |\hat p_x - p_x | = \ee\sup_{x \in I_k} |\hat p_x - p_x | $, we have
$$
\pp\left( \sup_{x \in I_k} | \hat p_x - p_x | >  \mu_k+t \right) \leq \exp\left( - \frac{1}{12} \left( \min \big(\frac{nt^2}{(\sigma^2+2\mu_k)}, nt\big)\right)\right).
$$
\end{corollary}
\begin{proof}
For any $x \not \in \mathbb Q$ (and in $I_k$), we have by definition of $\hat p_x$ and since $F$ (and thus $p$) is cadlag
\begin{align*}
|\hat p_x - p_x| = \lim_{y \in \mathbb Q, y>x, y \rightarrow x} |\hat p_y - p_y|.
\end{align*}
This above equality implies both 
$$\mu_k :=  \ee\sup_{x \in I_k \bigcap \mathbb Q} |\hat p_x - p_x | =  \ee\sup_{x \in I_k} |\hat p_x - p_x |,$$
and
\begin{align*}
\pp\left( \sup_{x \in I_k} | \hat p_x - p_x | >  \mu_k+t \right) &= \pp\left( \sup_{x \in I_k \bigcap \mathbb Q} | \hat p_x - p_x | >  \mu_k+t \right)\\ 
&\leq \exp\left( - \frac{1}{12} \left( \min \big(\frac{nt^2}{(\sigma^2+2\mu_k)}, nt\big)\right)\right).
\end{align*}
This concludes the proof.
\end{proof}

In order to use Corollary \ref{cor:wvv}, we want to bound $\mu_k$, and the following result proves the upper bound for $\mu_k$. 
\begin{lemma}\label{expectsup}
There exists a constant $D_1$ that depends only on $\beta_1$, on a lower bound on $\tau$ and on an upper bound on $C,C'$, and that is such that 
$$\mu_k = \ee \sup_{x \in I_k} | \hat p_x-p_x|  \leq D_1 \sqrt{\frac{e^{-k \tau} }{n}}.
$$
\end{lemma}

Let $k$ be such that $\delta_k >0$ and $\delta_k \leq \exp(-D_1^2/12)$. By plugging the result of Lemma~\ref{expectsup} into Corollary~\ref{cor:wvv}, with $t =12 \left( \sqrt{\frac{e^{- k \tau} \log(1/\delta_k)}{n}} + \frac{\log(1/\delta_k)}{n} \right)$, we obtain the inequality
\begin{equation}\label{largedev1}
\pp \left( \sup_{x \in I_k} | \hat p_x - p_x | > D_2 \left( \sqrt{\frac{e^{-k\tau} \log (1/\delta_k)}{n}} + \frac{\log(1/\delta_k)}{ n}\right) \right) \leq \delta_k,
\end{equation}
where $D_2 =\max(12, D_1)$.

By multiplying by $x^{\tau}$ (since on $I_k$, $x^\tau\leq e^{(k+1)\tau}$) in the probability in (\ref{largedev1}), we have 
\begin{equation*}
\pp \left( \sup_{x \in I_k} x^{\tau} | \hat p_x - p_x | >  eD_2 \left( \sqrt{\frac{e^{k\tau} \log (1/\delta_k)}{n}} + \frac{e^{k\tau}\log(1/\delta_k)}{ n}\right)  \right) \leq \delta_k.
\end{equation*}

Let $\Delta>0$. Set $\delta_k =\frac{\Delta \exp(-\exp((K-k)\tau))}{E}$, where $E = \sum_{k=1}^{K} \exp(-\exp((K-k)\tau))<\infty$ and depends on $\tau$ only (since it is a hypergeometric sum). Plugging this $\delta_k$ in the last inequality yields
\begin{align*}
\pp \Big( \sup_{x \in I_k} x^{\tau} | \hat p_x - p_x | &>  eD_2 \Big( \sqrt{\frac{e^{K\tau} \log (E/\Delta)}{n}} + \frac{e^{K\tau}\log(E/\Delta)}{ n}\Big)   \Big)\\ 
&\leq \frac{\Delta \exp(-\exp((K-k)\tau))}{E},
\end{align*}
which implies by definition of $K$ and by denoting $\zeta:=n^{-\frac{\beta_1}{2\beta_1+1}} \sqrt{\log (E/\Delta)} +n^{-\frac{2\beta_1}{2\beta_1+1}} \log (E/\Delta)$
\begin{align*}
\pp \Big( \sup_{x \in I_k} x^{\tau} | \hat p_x - p_x | &>  2eD_2 \zeta  \Big) \leq \frac{\Delta \exp(-\exp((K-k)\tau))}{E}.
\end{align*}

By combining the last equation for all $k=1, \ldots, K$, we obtain
\begin{equation*}
\pp \left( \sup_{x \leq n^{\frac{1}{\tau(2\beta_1+1)}}} x^{\tau} | \hat p_x - p_x | >  2eD_2\zeta   \right) \leq \Delta.
\end{equation*}
This concludes the proof of Lemma \ref{lem:lardev}.
\end{proof}

\begin{proof}[Proof of Lemma \ref{expectsup}]

Here we use the same notation for $B, K, I_k$ used in the proof of Lemma \ref{lem:lardev}. Let $k \leq K$. Consider the grid of points of $I_k$, that we write $\chi_k := (x_1, \ldots, x_i,\ldots, x_{\Upsilon_k-1})$, that are rationals, and that are such that $p_1 = e^k$ (or arbitrarily close rational to $e^k$) and otherwise
\begin{align*}
p_{x_i} - p_{x_{i+1}} = c/n,
\end{align*}
(or arbitrarily close rationals that verify this) until we reach $e^{k+1}$, and where $c$ is the smallest constant larger than $1$ such that $\log(\Upsilon_k)$ is an integer. 

\vspace{5pt}
\textbf{Step 1.} We claim that for any $x,y \in \chi_k^2$ such that $x \leq y$,  the tail probability of $\hat p_x - \hat p_y$ is upper bounded by an sub-exponential bound plus a sub-Gaussian bound with a distance function $d(x,y) =\sqrt{ \frac{p_x-p_y}{n}}$. That is, by denoting $\mathcal{U}_{x,y} =  \hat p_x- \hat p_y-(p_x-p_y)$,
$$
P\big(|\mathcal{U}_{x,y}| \geq u\big) \leq 2\exp\left(-\frac{u}{d(x,y)} \right) + 2\exp\left(-\frac{1}{2}\frac{u^2}{d(x,y)^2} \right).
$$

Note that $\hat p_x -\hat p_y = \frac{1}{n}\sum_{i=1}^n \one\{ x \leq X_i \leq y\}$ is the average of Binomial random variable with parameters $(n, p_x- p_y)$ where $p_x - p_y \geq 1/n$ since $x,y$ are points of $\chi_k$.
Then, by Bernstein inequality,
\begin{align*}
P(| \mathcal{U}_{x,y}  | \geq u) &\leq 2\exp \left(-\frac{1}{2} \frac{nu^2}{(p_x-p_y) + \frac{u}{3}}\right) \\
 &\leq 2\exp \left(-\frac{3}{2} nu\right) + 2\exp \left(-\frac{1}{2} \frac{nu^2}{(p_x-p_y)}\right).
\end{align*}
This implies (since $p_x - p_y \geq 1/n$)
\begin{align}\label{eq:subexpsubgauss}
P\left( |\mathcal{U}_{x,y} | \geq u \right) \leq 2\exp\left( - \frac{\sqrt{n}u}{\sqrt{p_x-p_y}}\right) + 2\exp \left(-\frac{1}{2} \frac{nu^2}{(p_x-p_y)}\right).
\end{align}
Thus, the claim is proved.

\vspace{5pt}
\textbf{Step 2.} In this step, we bound the expectation of the supremum when $x$ and $y$ can take possible $m$ number of values $x_j$ and $y_j$ in $\chi_k$, that is, $\ee\big(\sup_{j \leq m} | \hat p_{x_j} - \hat p_{y_j}-(p_{x_j}-p_{y_j})| \big)$ such that $p_x - p_y \leq d^2$.

Equation~\eqref{eq:subexpsubgauss} implies in particular that for $(x,y)\in \chi_k^2$, we can express $\mathcal{U}_{x,y}$ as a sum of a sub-Gaussian random variable $U_{x,y}$ plus a sub-exponential random variable $V_{x,y}$, i.e.~$\mathcal{U}_{x,y} = U_{x,y} + V_{x,y}$ and that are such that
\begin{align*}
\| U_{x,y}\|_{\Psi_2} \leq D_3\sqrt{\frac{p_x-p_y}{n}}, \ \ \ \|V_{x,y}\|_{\Psi_1} \leq D_4\sqrt{\frac{p_x-p_y}{n}}
\end{align*}
where $\| \cdot \|_{\Psi_1}$ and $\| \cdot \|_{\Psi_2}$ are the Orlicz norms $1$ and $2$.

Consider $m$ pairs $(x_j,y_j)\in \chi_k^2$ such that $p_x - p_y \leq d^2$. By definition of the Orlicz norms,
\begin{align*}
\|\sup_{j \leq m} |U_{x_j,y_j}|\|_{\Psi_2} \leq D_3\Psi_2^{-1}(m)\sqrt{\frac{p_x-p_y}{n}}\leq D_3\sqrt{2\log(m) + 1}\sqrt{\frac{d^2}{n}},
\end{align*}
and
\begin{align*}
\|\sup_{j \leq m} |V_{x_j,y_j}|\|_{\Psi_1} \leq D_4\Psi_1^{-1}(m)\sqrt{\frac{p_x-p_y}{n}} \leq D_4(\log(m) + 1)\sqrt{\frac{d^2}{n}}.
\end{align*}

This implies that
\begin{align*}
\mathbb P \Big( \sup_{j \leq m} |U_{x_j,y_j}| \geq u \Big) \leq \exp\left(-u^2\frac{n}{2D_3^2d^2(2\log(m) + 1)}\right),
\end{align*}
and
\begin{align*}
\mathbb P \Big( \sup_{j \leq m} |V_{x_j,y_j}| \geq u \Big) \leq \exp\left(-u\sqrt{\frac{n}{d^2}}\frac{1}{D_4(\log(m) + 1)} \right).
\end{align*}

These two equations give the following results, using $\mathbb E X = \int_0^\infty \mathbb P(X \geq u) du$, that
\begin{align*}
\mathbb E \Big(\sup_{j \leq m} |U_{x_j,y_j}| \Big) &\leq \int_0^{\infty}\exp\left(-u^2\frac{n}{2D_3^2d^2(2\log(m) + 1)} \right)du\\
&=\frac{\sqrt{\pi}}{2} \sqrt{\frac{2D_3^2d^2(2\log(m) + 1)}{n}} <1.5 \sqrt{\frac{D_3^2d^2(2\log(m) + 1)}{n}},
\end{align*}
and
\begin{align*}
\mathbb E\Big( \sup_{j \leq m} |V_{x_j,y_j}| \Big) &\leq \int_0^{\infty} \exp \left(-u\sqrt{\frac{n}{2d^2}}\frac{1}{D_4(\log(m) + 1)}\right) du\\
&=D_4 (\log(m) + 1)\sqrt{\frac{2d^2}{n}}<1.5D_4(\log(m) + 1)\sqrt{\frac{d^2}{n}}.
\end{align*}

Combining the above ideas, we have
\begin{align}
\mathbb E \Big(\sup_{j \leq m} |\hat p_{x_j}-\hat p_{y_j}-(p_{x_j}-p_{y_j})| \Big) &\leq \mathbb E \Big(\sup_{j \leq m} |U_{x_j,y_j}| \Big) + \mathbb E \Big(\sup_{j \leq m} |V_{x_j,y_j}| \Big) \nonumber\\
&\leq 3D_4 (2\log(m) + 1)\sqrt{\frac{d^2}{n}}. \label{eq:supcoucou}
\end{align}

\textbf{Step 3.} Now, we bound $\ee \sup_{x \in \chi_k}| \hat p_x - p_x|$ using the results in Step 2 by a chaining argument. For any $1 \leq i\leq \log_2(\Upsilon_k)$, we define chaining set $A_i$ by a sequence of finite subsets
$$A_i = \{x_{\lfloor j \Upsilon_k/2^i  \rfloor}, \hspace{2mm} \text{for} \hspace{2mm} j\in \mathbb N, \ 0<j < 2^i\}.$$
$A_1$ contains only one element (e.g. if $\Upsilon_k/2$ is an integer, $A_1 = \{x_{\Upsilon_k/2}\}$), and
the cardinality of $A_i$ is $2^i$. Note that $A_i \subseteq \chi_k$, and the last set $A_{\log_2(\Upsilon_k)}$ becomes $\{x_1, \ldots, x_{\Upsilon_k-1}\} =: \chi_k$.
 Also by definition of these sets, for any point $x \in \chi_k$, there exists a chain $(y_1, \ldots, y_{\Upsilon_k})$ such that $y_i\in A_i$ and 
$$|p_{y_{i+1}}-p_{y_i}| \leq \frac{|p_{e^k} - p_{e^{k+1}}|}{2^{i+1}} \leq (C+C') \exp(-k\tau) 2^{-i},$$
and such that $x = y_{\Upsilon_k}$. Note that given $y_{i+1}$, there is only two choices of $y_i$ which are possible (because of the previous equation). Let us write $\BB_{i+1}$ for such possible pairs (i.e.~that are at a distance less than $|p_{e^k} - p_{e^{k+1}}|/ 2^{i+1}$), and by definition there are less than $2\times 2^{i+1}$ such pairs.

By using the triangle inequality,
$|\hat p_x - p_x|\leq |\hat p_{y_1} - p_{y_1}| + \sum_{2 \leq i \leq \Upsilon_k} |\hat p_{y_{i+1}}-\hat p_{y_i} - (p_{y_{i+1}}-p_{y_i})|.$ It follows that by denoting $\mathcal{U}_i = \hat p_{y_{i+1}}-\hat p_{y_i} - (p_{y_{i+1}}-p_{y_i})$,
\begin{align*}
\mathbb E \big[ \sup_{x\in \chi_k} |\hat p_x - p_x|\big]&\leq  \mathbb E \sup_{\text{\scriptsize chain} \hspace{1mm}(y_1, \ldots, y_{\Upsilon_k})} \Big(|\hat p_{y_1} - p_{y_1}| + \sum_{1 \leq i \leq \Upsilon_k} |\mathcal{U}_i|\Big)\\
&\leq \mathbb E\sup_{y_1 \in A_1}|\hat p_{y_1} - p_{y_1}| + \sum_{1 \leq i \leq \Upsilon_k}\mathbb E \sup_{(y_{i+1}, y_i) \in \BB_{i+1}} |\mathcal{U}_i|.
\end{align*}

Using Equation~\eqref{eq:supcoucou} on $\BB_{i+1}$ which contains at most $2\times 2^{i+1}$ pairs satisfying $p_{y_{i+1}}-p_{y_i} \leq (C+C') \exp(-k\tau) 2^{-i}$, we get
\begin{align*}
\ee \Big[ \sup_{(y_{i+1}, y_i) \in \BB_{i+1}} |\mathcal{U}_i| \Big]&\leq 3D_4 (2(i+2) + 1)\sqrt{\frac{(C+C') \exp(-k\tau) 2^{-i}}{n}}\\
&\leq 9D_4\sqrt{C+C'}\sqrt{\frac{ \exp(-k\tau)}{n}} (i+2) 2^{-i/2}.
\end{align*}

Also since there is only one element in $A_1$, we have
$$ \mathbb E \sup_{y_1 \in A_1}|\hat p_{y_1} - p_{y_1}| \leq \mathbb E|\hat p_{y_1} - p_{y_1}|\leq \sqrt{\mathbb E (\hat p_{y_1} - p_{y_1})^2} \leq \sqrt{\frac{(C+C') \exp(-k\tau)}{n}}.$$

By plugging the above both equations in the chaining equation, we obtain
\begin{align*}
\mathbb E \sup_{x\in \chi_k} |\hat p_x - p_x| &\leq \sqrt{\frac{(C+C') \exp(-k\tau)}{n}} \\
& \ \ \ \ \ \  + \sum_{1 \leq i \leq \Upsilon_k} 9D_4\sqrt{C+C'}\sqrt{\frac{ \exp(-k\tau)}{n}} (i+2) 2^{-i/2}\\
&\leq \sqrt{C+C'}\sqrt{\frac{ \exp(-k\tau)}{n}} \left(1+9D_4\sum_{i=1}^{\infty} (i+2) 2^{-i/2} \right)\\
&\leq D_5\sqrt{C+C'}\sqrt{\frac{ \exp(-k\tau)}{n}},
\end{align*}
where $D_5 =  \left(1+9D_4\sum_{i=1}^{\infty} (i+2) 2^{-i/2} \right)<\infty$.

\textbf{Step 4.} In this final step, we extend the above inequality to any $x \in I_k$ not necessarily on the grid point.

Note that for any $x \in I_k$, there exist $\underline x, \bar x \in \chi_k^2$ such that $\underline x \leq x \leq \bar x$, $p_{\bar x} \leq p_x \leq p_{\underline x}$ and $\hat p_{\bar x} \leq \hat p_{x} \leq \hat p_{\underline x}$ where $p_{\underline x}-p_{\bar x} = c/n$.
Then for any $x \in I_k$, 
\begin{align*}
\mathbb E |\hat p_x - p_x| &\leq \mathbb E \Big((\hat p_{\underline x}-p_{\bar x})\one\{ \hat p_{x} \geq p_{x} \} + (p_{\underline x}- \hat p_{\bar x})\one\{ \hat p_{x} \leq p_{x} \}  \Big)\\
&\leq \mathbb E |\hat p_{\underline x}-p_{\bar x}| + \mathbb E |p_{\underline x}- \hat p_{\bar x}| \\
&\leq \mathbb E |\hat p_{\underline x}-p_{\underline x}| + \mathbb E |p_{\bar x}- \hat p_{\bar x}|  + 2\mathbb E |p_{\bar x} - p_{\underline x} | \\
&\leq 2D_5\sqrt{C+C'}\sqrt{\frac{ \exp(-k\tau)}{n}} + \frac{2c}{n}\\
&\leq 4D_5\sqrt{C+C'}\sqrt{\frac{ \exp(-k\tau)}{n}},
\end{align*}
where the last inequality is followed since for $k \leq K$, we know that $\exp(-k\tau) \geq \exp(-K\tau) = n^{-1/(2\beta_1+1)} \geq 1/n$. This concludes the proof.
\end{proof}

\subsection{Proof of Theorem~\ref{thm:test2}}

Let $X_1, \ldots, X_n$ be an i.i.d. random sample from a distribution $F \in \mathcal{S}_1$.

Let $\hat \tau$ be an estimator of $\tau$ such that for any $\tau\in \II_1$, we have with probability at least $1-\eta$ 
\begin{equation}\label{taulargedev}
|\hat \tau - \tau| \leq n^{-\frac{\beta_1}{2\beta_1+1}}c_1(\eta),
\end{equation}
where $c_1$ is a function defined on $(0,1)$.
For instance, Theorem 1 in \citep{cheng2001} implies that with Hill estimator $\hat \tau_H$, we can choose
(asymptotically) $c_1(\eta)$ as $q_{1-\eta/2} \hat \tau_H$, where $q_{1-\alpha/2}$ is such as $\mathbb P(|\mathcal N(0,1)| \geq q_{1-\alpha/2}) = \alpha$ (where $\mathcal N$ is the standard Gaussian distribution). See also Theorem 3.6 and Remark 3.7 of~\citet{carpentierkim} for another estimator for which $c_1(\eta) \sim \sqrt{\log(1/\eta)}$ is well defined with a finite $n$.

Also we define $\hat C$ as an estimator of $C$ such that for any $C\in \II_2$, we have with probability at least $1-\eta$ 
\begin{equation}\label{Clargedev}
|\hat C - C| \leq \log(n)n^{-\frac{\beta_1}{2\beta_1+1}}c_2(\eta),
\end{equation}
where $c_2$ is a function defined on $(0,1)$. 

For instance, we can define $\hat C$ as follows, 
$$
\hat C = n^{\frac{1}{2\beta_1+1}}  \hat p_{\hat B},
$$
where  
\begin{equation}\label{eq:cc2}
\hat B = n^{1/  \hat \vartheta}, \ \  \hat \vartheta = (\hat \tau + n^{-\frac{\beta_1}{2\beta_1+1}}c_1(\eta))(2\beta_1+1).
\end{equation}
From (\ref{taulargedev}), for a sufficiently large $n$ (such that $2\log(n)n^{-\frac{\beta_1}{2\beta_1+1}}c_1(\eta)/\tau \leq 1/2$, for any $\tau \in \mathcal I_1$)), we know with probability $1-\eta$, 
\begin{equation}\label{rangeBhat}
\frac{1}{2} n^{\frac{1}{\tau(2\beta_1+1)}} \leq \left(1 - \frac{2\log(n)n^{-\frac{\beta_1}{2\beta_1+1}}c_1(\eta)}{\big(\tau(2\beta_1+1)\big)} \right) n^{\frac{1}{\tau(2\beta_1+1)}} \leq \hat B \leq n^{\frac{1}{\tau(2\beta_1+1)}}.
\end{equation}
\color{black}
To prove such $c_2(\eta)$ exists in (\ref{Clargedev}), we first split $\hat C -C$ into the two summations,
$$\hat C - C =  n^{\frac{1}{2\beta_1+1}}  \hat p_{\hat B} - C
= n^{\frac{1}{2\beta_1+1}} (\hat p_{\hat B} - p_{\hat B}) + (n^{\frac{1}{2\beta_1+1}}p_{\hat B}-C) =:(*)+(**).$$
Then, using Lemma \ref{lem:lardev} and (\ref{rangeBhat}), with probability $1-2\eta$,
$$(*) := n^{\frac{1}{2\beta_1+1}} (\hat p_{\hat B} - p_{\hat B}) \leq 2Dn^{-\frac{\beta_1}{2\beta_1+1}}\log(1/\eta). $$
The second term $(**)$ can be bounded by the definition of the second order Pareto distributions,
\begin{align}
(*)&= n^{\frac{1}{2\beta_1+1}}p_{\hat B}-C \leq
n^{\frac{1}{2\beta_1+1}} \Big(C \hat B^{-\tau} + C' \hat B^{-\tau (\beta_1+1)} \Big) - C \nonumber \\
&=  C \left( n^{\frac{1}{2\beta_1+1}} \hat B^{-\tau} -1\right) + C' n^{\frac{1}{2\beta_1+1}}\hat B^{-\tau(\beta_1+1)}  \nonumber \\
&\leq  C \Big( n^{\frac{ n^{-\frac{\beta_1}{2\beta_1+1}}2c_1(\eta)/\tau}{2\beta_1+1}} -1 \Big) + C' n^{-\frac{\beta_1}{2\beta_1+1}} n^{\frac{n^{-\frac{\beta_1}{2\beta_1+1}} 2c_1(\eta)}{\tau(2\beta_1+1)}} \label{calculation}\\
&\leq  2 (C+C')\log(n) \frac{ n^{-\frac{\beta_1}{2\beta_1+1}}c_1(\eta)/\tau}{2\beta_1+1}, \label{calculation2}
\end{align}
for $n$ large enough so that $2n^{-\frac{\beta_1}{2\beta_1+1}}c_1(\eta)/\tau \leq 1/2$, where the first term in (\ref{calculation}) is obtained with probability $1-\eta$ as follows,
\begin{align*}
n^{\frac{1}{2\beta_1+1}} \hat B^{-\tau}&\leq n^{\frac{1}{2\beta_1+1}} n^{\frac{-\tau}{\big(\tau + 2n^{-\frac{\beta_1}{2\beta_1+1} }c_1(\eta)\big) (2\beta_1+1)}}\\
&\leq n^{\frac{1}{2\beta_1+1}} n^{-\frac{1}{2\beta_1+1} \big(1-\frac{n^{-\frac{\beta_1}{2\beta_1+1} }2c_1(\eta)}{\tau} \big)}.
\end{align*}
The second term in (\ref{calculation}) is upper bounded similarly. Then (\ref{calculation2}) is followed by the taylor expansion.
Using the exact same ideas for the lower bound of $\hat C-C$, we have proved (\ref{Clargedev}) with probability $1-\eta$,
 \begin{equation}\label{c2function}
 c_2(\eta) = 2D\log(3/\eta)  + 2(C+C')\frac{c_1(\eta/3)/\tau}{2\beta_1+1}.
 \end{equation}

Using large deviation results from (\ref{taulargedev}) and Lemma \ref{lem:lardev}, we can obtain with probability at least $1-2\eta$, 
for any $x \leq n^{\frac{1}{\tau(2\beta_1+1)}} =: B$,
\begin{align}
|x^{\hat \tau} \hat p_x - & x^{\tau} \hat p_x| = x^{\tau} \hat p_x |x^{\hat \tau - \tau} - 1| \leq x^{\tau} \hat p_x|x^{n^{-\frac{\beta_1}{2\beta_1+1}}c_1(\eta)} - 1|\nonumber \\
&\leq 2x^{\tau} \hat p_x \log(x) n^{-\frac{\beta_1}{2\beta_1+1}}c_1(\eta)\nonumber \\
&\leq \frac{4}{\tau(2\beta_1+1)} (C+C' + D \log(1/\eta)) \log(n) n^{-\frac{\beta_1}{2\beta_1+1}}c_1(\eta)=:(\star). \label{res1}
\end{align}

Large deviation property for $\hat \tau$ in (\ref{taulargedev}) implies also that with probability at least $1-\eta$, for  any $x \leq n^{\frac{1}{\tau(2\beta+1)}}$, we have
\begin{align}
|C' x^{-\hat \tau\beta_0} - C' x^{-\tau\beta_0}| &\leq C'x^{-\tau\beta_0} | x^{(\tau -\hat \tau)\beta_0} - 1| \nonumber\\
&\leq C'x^{-\tau\beta_0} | x^{n^{-\frac{\beta_1}{2\beta_1+1}}c_1(\eta)\beta_0} - 1| \nonumber\\
&\leq C'x^{-\tau\beta_0} \log(x)n^{-\frac{\beta_1}{2\beta_1+1}}c_1(\eta)\beta_0 \nonumber\\
&\leq \frac{C'}{\tau(2\beta_1+1)}\log(n)n^{-\frac{\beta_1}{2\beta_1+1}}c_1(\eta)\beta_0. \label{Cprimelast}
\end{align}

Combining these equations (\ref{res1}), (\ref{Cprimelast}), and (\ref{Clargedev}), with probabibility at least $1-3\eta$, we have
\begin{align*}
&\Big|\big(|x^{\hat \tau} \hat p_x - \hat C| - C' x^{-\hat \tau\beta_0}\big) - \big(|x^{\tau} \hat p_x - C| - C' x^{-\tau\beta_0}\big) \Big|\\
&\leq (\star) + \frac{C'\beta_0c_1(\eta)}{\tau(2\beta_1+1)} \log(n)n^{-\frac{\beta_1}{2\beta_1+1}} + 
 \log(n)n^{-\frac{\beta_1}{2\beta_1+1}}c_2(\eta)\\
&\leq E(\eta)\log(n)n^{-\frac{\beta_1}{2\beta_1+1}},
\end{align*}
where 
$$
E(\eta) = \frac{4(C+C' D \log(1/\eta)) c_1(\eta)}{\tau(2\beta_1+1)} +\frac{C' \beta_0c_1(\eta)}{\tau(2\beta_1+1)}+  c_2(\eta).
$$
This implies, together with Lemma~\ref{lem:lardev}, that for any $x \leq n^{\frac{1}{\tau(2\beta_1+1)}}$ with probabibility at least $1-3\eta$
\begin{align}
&\Big|\big(|x^{\hat \tau} \hat p_x - \hat C| - C' x^{-\hat \tau\beta_0}\big) - \big(|x^{\tau} p_x - C| - C' x^{-\tau\beta_0}\big) \Big|\nonumber\\
&\leq (E(\eta)\log(n) + D\log(1/\eta) )n^{-\frac{\beta_1}{2\beta_1+1}}.\label{eq:coucouhaha}
\end{align}

  Based on these previous results, with $\hat \tau$ satisfying (\ref{taulargedev}), $\hat C$ satisfying (\ref{Clargedev}), and $\hat B$ as in (\ref{eq:cc2}), \color{black} we propose the following test statistic
\begin{equation}\label{teststat}
T_n = \sup_{x \leq \hat B} \Big(|x^{\hat \tau} \hat p_x - \hat C| - C' x^{-\hat \tau\beta_0} \Big).
\end{equation}
The test is of the form (similar to the case in Theorem \ref{thm:test1})
\begin{align*}
\Psi_n = \one\{ T_n \geq \rho_n/2\},
\end{align*}
where 
\begin{equation}\label{rhondefinition}
 \rho_n \geq \max\left(2(E(\eta)\log(n)+D\log(1/\eta)), 2C' \right) n^{-\frac{\beta_1}{2\beta_1+1}}.
\end{equation} 
Recall that $H_0 : F \in \bigcup_{\tau \in \II_1, C \in \II_2} \set_0(\tau, C)$ and $H_1 : F \in \cup_{ \tau \in \II_1, C \in \II_2} \tilde \set_1 (\tau, C, \rho_n)$. Again, we reject the null if $\Psi_n = 1$, and vice versa.

Set $\rho_n \geq 2C' n^{-\frac{\beta_1}{2\beta_1+1}}$. By Lemma~\ref{lem:doublehypo}, we know that
\begin{itemize}
\item[(i)] $F \in H_0$ implies $\sup_x \Big(|x^{\tau} p_x - C| - C' x^{-\tau\beta_0}\Big) \leq 0$.
\item[(ii)] $F \in H_1$ implies that $\sup_{x \leq  n^{\frac{1}{\tau(2\beta_1+1)}}} \Big(|x^{\tau} p_x - C| - C' x^{- \tau\beta_0}\Big) \geq \rho_n$.
\end{itemize}
This implies together with Equation~\eqref{eq:coucouhaha} that with probabibility at least $1-4\eta$
\begin{itemize}
\item[(i)] $F \in H_0$ implies $T_n \leq \sup_{x\leq n^{\frac{1}{\tau(2\beta_1+1)}}} \Big(|x^{\hat \tau} \hat p_x - \hat C| - C' x^{-\hat \tau\beta_0}\Big) < \rho_n/2$.
\item[(ii)] $F \in H_1$ implies that $T_n \geq \sup_{x \leq \frac{1}{2} n^{\frac{1}{\tau(2\beta_1+1)}}} \Big(|x^{\hat \tau} \hat p_x - \hat C| - C' x^{- \hat \tau\beta_0}\Big) \geq \rho_n/2$.
\end{itemize}
Using Equation~\eqref{rangeBhat}, we have $$\sup_{F \in H_0} \pp_F \Psi_n + \sup_{F \in H_1} (1-\Psi_n) \leq 1-9\eta.$$
This concludes the proof by setting $\alpha = 9\eta$.

\subsection{Proof of Theorem \ref{thm:CI2}} \label{subsec:CI2}

\paragraph{\bf{Proof of [A.]}}
We use the same $\eta$ and test statistics $T_n$, and test $\Psi_n$ from the previous section. 

Denote $G(\beta) := c_1(\eta)n^{-\beta/(2\beta+1)}$ where $c_1(\eta)$ is defined such that for a sample from $F \in H_i$, we have $|\hat \tau - \tau| \leq c_1(\eta)n^{-\beta_i/(2\beta_i+1)}$ (when $i=0$ or $i=1$).
Now we consider the confidence interval based on the test: 
$$
C_n = \left\{ \tau' : | \hat \tau - \tau' | \leq G(\beta_0) (1-\Psi_n) + G(\beta_1) \Psi_n \right\},
$$

Note that $\beta_0 > \beta_1$ means $G(\beta_0) < G(\beta_1)$, and recall that 
$$
\PP_n = \{ F : F \in H_0 \cup H_1\} = \Big(\bigcup_{\tau \in \II_1, C \in \II_2}\set_0(\tau,C)\Big) \bigcup \Big(\bigcup_{\tau \in \II_1, C \in \II_2}\tilde \set_1(\tau,C, \rho_n)\Big).
$$

First, under the null, 
\begin{align*}
\sup_{F \in \bigcup_{ \tau >0, C>0} \set_0(\tau,C) \bigcap \PP_n} \pp_F(|C_n| > G(\beta_0)) &\leq
\sup_{F \in \bigcup_{ \tau \in \II_1, C \in \II_2} \set_0(\tau,C)} \pp_F(\Psi_n=1)\\
&\leq 4\eta=\frac{4\alpha}{9}
\end{align*}
by definition of $C_n$.

Second, under the alternative,
$$
\sup_{F \in \bigcup_{ \tau >0, C>0} \set(\tau, \beta_1, C, C') \bigcap \PP_n} \pp_F(|C_n| > G(\beta_1)) 
=0.
$$

The third condition in Definition of \ref{def:CI} is shown using the last calculation in the proof of Theorem 2,
\begin{align*}
\inf_{F \in \PP_n} \pp_F (\tau \in C_n)  &\geq \min\left( \inf_{F \in H_0} \pp_F(\tau \in C_n), \inf_{F \in H_1} \pp_F(\tau \in C_n)\right)\\
&\geq\min \Big( \inf_{F \in H_0} \pp_F\left( \hat \tau \in \tau \pm G(\beta_0)\right) \pp_F(\Psi_n =0), \\
&\hspace{40pt}\inf_{F \in H_1} \pp_F\left( \hat \tau \in \tau \pm G(\beta_1)\right)\pp_F(\Psi_n = 1) \Big) \\
&\geq 1-5\eta = 1-\frac{5\alpha}{9}.
\end{align*}

Thus, we have proved the existence of an adaptive and uniform confidence interval for $\PP_n$ (by checking the two conditions (\ref{def2.2.1}) and (\ref{def2.2.3})).

\paragraph{\bf{Proof of [B.]}}

The proof depends on the lower bound construction, which is previously considered similarly in the papers~\citep{drees2001,novak2013,carpentierkim}.

Let $n \geq 2$. Let $\tau>0$, $\upsilon>0$, $\beta_0>\beta_1 >0$, and we define
$ B = n^{\frac{1}{\tau(2\beta_1 + 1)}}$, $t  =  \upsilon B^{-\tau \beta_1} =  \upsilon n^{-\frac{\beta_1}{2\beta_1 + 1}}$, $\tau_0 =\tau$, and $\tau_1 = \tau-t  = \tau -  n^{-\frac{\beta_1}{2\beta_1+1}}$.
Then we consider the Pareto with $\tau$ parameter as $F_0$ such as $1- F_0(x) = x^{-\tau}$, and for $F_1$ we perturb the tail larger than $B$ so that it has heavier tail. That is, we let $1- F_1(x) = x^{-\tau} \one \{ 1\leq x \leq B\} + B^{-t}x^{-\tau + t} \one \{ x > B\}$.

Note that $F_0 \in \set_0(\tau, 1) \subseteq \set_0(\tau_1, B^{-t})$ and $F_1 \in \tilde  \set_1(\tau_1, B^{-t}, \rho_n)$. Let $\delta>0$. Then it is known \citep[as proved in][]{drees2001,novak2013} that there exists no $\delta-$uniformly consistent test for distinguishing between $F_0$ and $F_1$ whenever $n$ is large enough, for small enough $\upsilon$.

Now, we use a contradiction to prove our claim. Suppose $0<\alpha<1/3$ and $3\alpha = \delta$. Assume that there exists an $\alpha$-uniform and adaptive confidence interval $C_n$ for the first order parameter when $\mathcal P_n = \{F_0, F_1\}$. Then we consider the test $\Psi_n$ such that
\begin{align*}
\Psi_n = 1 - \one\{\tau \in C_n\}\one\{|C_n| \leq D n^{-\frac{\beta_0}{2\beta_0+1}}\}.
\end{align*}

Then since $C_n$ is uniform and adaptive, we have
\begin{align*}
\ee_{F_0} (\Psi_n) &\leq \ee_{F_0} \one\{\tau \not\in C_n\} + \ee_{F_0} \one\{|C_n| > D n^{-\frac{\beta_0}{2\beta_0+1}}\}\\
&\leq 2\alpha.
\end{align*}

Also we have
\begin{align*}
\ee_{F_1} (1-\Psi_n) &\leq \ee_{F_1} \one\{\tau \in C_n\}\one\{|C_n| \leq D n^{-\frac{\beta_0}{2\beta_0+1}}\}\\
&\leq \ee_{F_1} \one\{\tau \in C_n\}\\
&\leq \alpha.
\end{align*}
This implies that $\Psi_n$ is $3\alpha$ uniformly consistent for $\mathcal P_n = \{F_0, F_1\}$. This contradicts the fact that no $\delta-$uniformly consistent test exists. This concludes the proof.

%

\subsection{Proof of the lower bound in Theorem~\ref{thm:test1} (Proof of [B.])}
Here, we prove the lower bound by constructing two distributions $F_0$ and $F_1$ in the model with the specific $\rho_n 
\sim n^{-\beta_1/(2\beta_1+1)}$ and by proving the distance between $F_0^n$ and $F_1^n$ is close enough so that these two are not distinguishable as $n \rightarrow \infty$ (so that $\alpha$-uniform consistent test does not exist).

Let $\tau, \beta_1>0$. Let $F_0$ be the distribution such that for any $x \geq 1$, we have
$$1 - F_0(x) = x^{-\tau}.$$ Note that $F_0 \in \set(\tau,\infty,1,0)$.


Let $\upsilon >0$ be a small constant. Now, we construct another continuous distribution $F_1$. Let 
$B = n^{1/(\tau(2\beta_1+1))}$, $t = \upsilon B^{-\tau \beta_1} = \upsilon n^{-\beta_1/(2\beta_1+1)}$, and let $
C'^2 \leq \frac{2\beta_1+1}{(\beta_1+1)^2} \frac{1}{3\tau^2}$.
Also we suppose that $n$ is large enough such that $t \leq \min( \frac{\sqrt{3}\upsilon \tau}{2\sqrt{2\beta_1+1}}, \frac{\tau}{4})$. 
Then, consider $B_1$ such that $B < B_1 = (1+\tilde C)B$ where $\tilde C >0$ (later it will be chosen as the smallest $\tilde C$ such that $F_1$ is continuous).
More precisely, 
\begin{align}\label{defF_1}
1-F_1(x) =& x^{-\tau} \one \{1 \leq x \leq B \} + B^{-t} x^{-\tau+t} \one\{ B<x<B_1\} \\ \nonumber
& + (x^{-\tau} + C' x^{-\tau(1+\beta_1)}) \one\{x \geq B_1\}.
\end{align}
As we can see in the definition (\ref{defF_1}), $F_1$ is defined to be slightly perturbed distribution from $F_0$ such that it is exactly Pareto with parameter $\tau$ on the region $x \leq B$, and it attains the upper bound for the second order Pareto tails after $B_1$, but in the middle region $B \leq x \leq B_1$ it only satisfies exactly Pareto with parameter $\tau-t$.


Equivalently, we are testing the following two hypotheses,
$$
H_0 : F=F_0 \quad \textbf{vs.} \quad H_1 : F = F_1,
$$
and show that there does not exist uniform consistent test.
Let $\beta_0 > \beta_1$. By definition, $F_0 \in \set(\tau, \beta_0, 1, C')$ (since it is exactly pareto). 

\vspace{5pt}
{\bf Step 1. Checking that $F_1 \in \set(\tau,\beta_1,1,C')$.}  Clearly, we only need to check the second order Pareto condition for the region $\{x: B <x<B_1\}$; we need to show that $B^{-t}x^{-\tau+t} \leq x^{-\tau} + C' x^{-\tau(1+\beta_1)}$, or, $B^{-t} x^t \leq 1+C' x^{-\tau \beta_1}$ for $\{x: B<x<B_1\}$.
Trivially the inequality is true when $x=B$. Also, since the LHS is an increasing function of $x$ while the RHS is a decreasing function of $x$, we verify the claim $F_1 \in \set(\tau,\beta_1,1,C')$ by choosing $B_1 = B+u$ such that
\begin{align}
B^{-t} B_1^t &= 1 + C' B_1^{-\tau \beta_1} \Leftrightarrow \left(1+\frac{u}{B}\right)^t -1 = C'(B+u)^{-\tau\beta_1}.  \label{condB}
\end{align}

\vspace{5pt}
{\bf Step 2. Range of $B_1$.} For convenience, we let $u =: \tilde C B$ (with $\tilde C >0$). Then from~(\ref{condB}),
\begin{equation}\label{ctildecond}
(\tilde C+1)^t =1+ C' (\tilde C+1)^{-\tau \beta_1} B^{-\tau \beta_1}=1+ C' (\tilde C+1)^{-\tau \beta_1} \frac{t}{\upsilon}
\end{equation}
which gives the upper  bound
\begin{equation}\label{rangeB}
 \log (\tilde C+1) \leq \frac{1}{\upsilon}C'(\tilde C+1)^{-\tau \beta_1} \leq \frac{1}{\upsilon}C'.
\end{equation}
Then, $B_1 = (1+\tilde C)B \leq \exp(C'/\upsilon)B$.

\vspace{5pt}
{\bf Step 3. Checking that Separation condition is verified.} 
First, we claim that
$$\|x^\tau(1-F_1(x)) - 1\|_\infty \geq M n^{-\frac{\beta_1}{2\beta_1+1}}.$$
Indeed, since the constructed $F_1$ is a continuous function, we just need to check the kink point $B_1$. Note that 
$$
B_1^{-\tau} \left( 1-F_1(B_1) \right)-1 = C' B_1^{-\tau \beta_1} = C' (\tilde C+1)^{-\tau \beta_1} B^{-\tau \beta_1} \geq C'e^{-\frac{C'\tau \beta_1}{\upsilon}}\frac{t}{\upsilon}
$$
where the last inequality is followed by definition of $B$ and the upper bound (\ref{rangeB}) for $\tilde C$.
Thus $||x^{-\tau} (1-F_1(x))-1||_\infty \geq M n^{-\beta_1/(2\beta_1+1)}$ for $M \leq C'e^{-\frac{C'\tau \beta_1}{\upsilon}}$.

This implies that there exists a point $x_0$ such that
$$|x_0^\tau(1-F_1(x_0)) - 1| \geq (M/2) n^{-\frac{\beta_1}{2\beta_1+1}}.$$
Consider a function $F \in \set(\tau, \beta_0, 1, C')$. We know that this function is such that
$$|x_0^\tau(1-F(x_0)) - 1| \leq C' n^{-\frac{\beta_0}{2\beta_0+1}}.$$
This implies in particular that
$$|x_0^\tau(1-F_1(x_0)) - x_0^\tau(1-F(x_0)) | \geq (M/2) n^{-\frac{\beta_1}{2\beta_1+1}} - C' n^{-\frac{\beta_0}{2\beta_0+1}} \geq (M/4) n^{-\frac{\beta_1}{2\beta_1+1}},$$
for $n$ large enough. This implies that $\|F_1 - \set(\tau, \beta_1, 1, C')\|_{\infty, \tau} \geq (M/4) n^{-\frac{\beta_1}{2\beta_1+1}}$, so $F_1$ belongs to the separated set $\tilde \set_1(\tau, \beta_1,C,C', \rho_n)$ with $\rho_n = M/4 n^{-\frac{\beta_1}{2\beta_1+1}}$.

\vspace{5pt}
{\bf Step 4. Computing the $KL$ divergence.} We define $\tilde F_1$ such that $1-\tilde F_1(x) = x^{-\tau} \one\{1 \leq x \leq B\} + B^{-t} x^{-\tau+t} \one\{x>B\}$ with the same $B = n^{1/(\tau(2\beta_1+1))}$ and $t = \upsilon B^{-\tau \beta_1}$.
From the same calculation in \citet[page 29]{carpentierkim}, we know that,
$$
KL(F_0,\tilde F_1) := \int f_0 \log \frac{f_0}{\tilde f_1} \leq \frac{\upsilon^2}{2\tau n}.
$$
Then, from the fact that $\tilde f_1 \neq  f_1$ only for $x \geq B_1$, it suffices to show that 
$$\int_{B_1}^{\infty} f_0 \log \frac{f_0}{f_1} \leq \int_{B_1}^{\infty} f_0 \log \frac{f_0}{\tilde f_1}, \ \ \text{or} \ \ 
\int_{B_1}^\infty f_0 \log \frac{f_1}{\tilde f_1} \geq 0. 
$$
Note that
$$
\frac{f_1(x)}{\tilde f_1(x)} = \frac{1+C' (1+\beta_1)x^{-\tau \beta_1}}{(1-\frac{t}{\tau})B^{-t}x^t}
$$
which gives 
\begin{align*}
\int_{B_1}^\infty f_0 \log \frac{f_1}{\tilde f_1}dx &=\int_{B_1}^\infty \tau x^{-\tau-1} \log \left( \big( \frac{B}{x} \big)^t \frac{\tau}{\tau-t} \right)dx \\
& \ \ \ + \int_{B_1}^\infty \tau x^{-\tau-1} \log \left( 1+ C'(1+\beta_1)x^{-\tau \beta_1} \right)dx \\
&=: (i) + (ii),
\end{align*}
where 
\begin{align}
(i) &= -t \int_{B_1}^\infty \tau x ^{-\tau-1} \log \left( \big( \frac{x}{B_1} \big) \big( \frac{\tau-t}{\tau} \big)^{1/t} \right) dx -t \int_{B_1}^\infty \tau x^{-\tau-1} \log \left(\frac{B_1}{B} \right) \nonumber\\
&= B_1^{-\tau} \left( \log \big( \frac{\tau}{\tau-t}\big) - \frac{t}{\tau}\right) - t B_1^{-\tau}\log \left(\frac{B_1}{B} \right) \label{bound}
\end{align}
where the second eqaulity is followed by the same calculation as in the paper by \citet[page 29]{carpentierkim}. Then we bound $\log(\tau/(\tau-t))-t/\tau$ below. Using $\log (1+u) \geq u-u^2/2$ for $0<u<1/2$, since $t/(\tau-t) \leq 1/2$ by $t \leq \tau/4$,
\begin{align}
\log \big( \frac{\tau}{\tau-t}\big) - \frac{t}{\tau} &\geq \frac{t}{\tau-t} -\frac{1}{2} \left( \frac{t}{\tau-t}\right)^2- \frac{t}{\tau} \nonumber \\
&= \frac{t^2}{\tau(\tau-t)} -\frac{1}{2} \left( \frac{t}{\tau-t}\right)^2
= \frac{t^2}{\tau-t} \left(\frac{1}{\tau} - \frac{1}{2(\tau-t)} \right). \label{t2bound}
\end{align}

Now we consider $(ii)$. Note that $x^{-\tau \beta_1}$ is decreasing in $x \geq B_1 = (\tilde C+1)B$. Again, using $\log(1+u) \geq u-u^2/2$ for $u\leq 1/2$, since $C'(\beta_1+1) B_1^{-\tau \beta_1} \leq C'(\beta_1+1) t/\upsilon \leq 1/2$ by the upper bound assumption for $t$ and $C'$,
\begin{align}
(ii):= &\int_{B_1}^\infty \tau x^{-\tau-1}  \log \left( 1+C'(\beta_1+1)x^{-\tau\beta_1}\right)dx \nonumber \\
&\geq C'(\beta_1+1) \int_{B_1}^\infty \tau x^{-\tau-1-\tau \beta_1} dx - \frac{C'^2(1+\beta_1)^2}{2} \int_{B_1}^\infty \tau x^{-\tau-1-2\tau\beta_1}dx \nonumber\\
&= C' B_1^{-\tau(\beta_1+1)} - \frac{C'^2(\beta_1+1)^2}{(2\beta_1+1)}  B_1^{-\tau-2\tau \beta_1} \nonumber\\
&=: C'B_1^{-\tau(\beta_1+1)} - A B_1^{-\tau-2\tau \beta_1}, \label{bound1}
\end{align}
by letting $A:= \frac{C'^2(\beta_1+1)^2}{(2\beta_1+1)}$.
Combining (\ref{bound}), (\ref{t2bound}), and (\ref{bound1}), and using the equality $\log \big(\frac{B_1}{B}\big) = \log (1+\tilde C)$
\begin{align}
\int_{B_1}^\infty &f_0 \log \frac{f_1}{\tilde f_1} = (i) + (ii) \nonumber \\
&\geq B_1^{-\tau} \left( \log\big( \frac{\tau}{\tau-t}\big)-\frac{t}{\tau}- t\log \big(\frac{B_1}{B}\big)+ C' B_1^{-\tau\beta_1} -AB_1^{-2\tau\beta_1}\right) \nonumber \\ 
&\geq B_1^{-\tau} t^2 \left( \frac{1}{\tau-t} \left(\frac{1}{\tau} - \frac{1}{2(\tau-t)} \right) -  A (1+\tilde C)^{-2\tau \beta_1}  \right), \label{cal1}
\end{align}
where the final inequality is obtained by observing by the equality (\ref{ctildecond})
$$t\log (\tilde C+1) = \log \left(1+C'(\tilde C+1)^{-\tau \beta_1}\frac{t}{\upsilon}\right) \leq C'(\tilde C+1)^{-\tau \beta_1}\frac{t}{\upsilon}.$$
Then, using $t \leq \tau/4$ and $(1+\tilde C)^{-2\tau \beta_1} \leq 1$,
\begin{align*}
\frac{1}{\tau-t} \left(\frac{1}{\tau} - \frac{1}{2(\tau-t)} \right) -  A (1+\tilde C)^{-2\tau \beta_1} &\geq \frac{1}{\tau}\frac{1}{3\tau} - \frac{C'^2(\beta_1+1)^2}{(2\beta_1+1)} \geq 0
\end{align*}
by our assumption that $C'^2 \leq \frac{2\beta_1+1}{(\beta_1+1)^2} \frac{1}{3\tau^2}$.

\vspace{5pt}
{\bf Step 5. Conclusion.} By using the result of the previous step, we have
$$
KL(F_0,F_1) \leq KL(F_0,\tilde F_1) := \int f_0 \log \frac{f_0}{\tilde f_1} \leq \frac{\upsilon^2}{2\tau n}.
$$
By definition of the KL divergence, we know
$$
KL(F_0^{n},F_1^{n}) = n KL(F_0,F_1) \leq \frac{\upsilon^2}{2\tau}.
$$
This implies by Pinsker's inequality that
$$
\max\Bigg( |\mathbb P_0(\Psi_n = 0) - \mathbb P_1(\Psi_n = 0)|, |\mathbb P_0(\Psi_n = 1) - \mathbb P_1(\Psi_n = 1)| \Bigg) \leq \sqrt{\frac{1}{4\tau }}\upsilon,
$$
which in turn implies (using $\pp_0(\Psi_n = 0) = 1-\pp_0(\Psi_n = 1)$ and $\pp_1(\Psi_n = 1) = 1-\pp_1(\Psi_n=0)$) that we have
$
 \mathbb P_0(\Psi_n = 1)+\mathbb P_1(\Psi_n = 0)   \geq 1 - \sqrt{\frac{1}{4\tau}}\upsilon.
$
This is equivalent to 
$$
\mathbb \mathbb \ee_0 (\Psi_n) + \mathbb E_1 (1-\Psi_n)  \geq 1 - \sqrt{\frac{1}{4\tau}}\upsilon,
$$
which implies that there exists no uniformly consistent test between $F_0$ and $F_1$.

This concludes the proof.

\subsection{Proof of Theorem~\ref{generaltesting}}\label{proof:generaltesting}


Let $M_n:= \lfloor \log(n)/\xi \rfloor$, and we define $M_n-1$ number of tests indexed by $i$ ranged from $0$ to $M_n-2$ such that
\begin{align*}
H_0(i): F \in \set(\tau,\beta_{i},C,C') \hspace{3mm} \mathbf{vs.} \hspace{3mm} H_1(i): F \in \tilde \set(\tau,\beta_{i+2},\beta_i,C,C' \rho_n(\beta_{i+2})),
\end{align*}
where $\tilde \set(\tau,\beta_{i+2},\beta_i,C,C' \rho_n(\beta_{i+2}))$ is defined in Equation (\ref{generalrhon}). 

Let $\alpha>0$ and we define the separation rate similarly to (\ref{rhondefinition}) but with $\alpha/(9M_n)$ as below, 
$$\rho_n(\beta_{i}) = \max \left\{2(E(\alpha/(9M_n))\log(n)+D\log((9M_n)/\alpha), 2C' \right\} n^{-\frac{\beta_{i}}{2\beta_{i}+1}}.$$
From the proof of Theorem \ref{thm:test2} (on page 20--22), we know that there exist estimators of $\tau$, $C$ such that $c_1(\eta) \sim \sqrt{\log(1/\eta)}$ and $c_2(\eta) \sim \log(1/\eta)$ thus we have $E(\eta) \sim (\log(1/\eta))^{3/2}$ uniformly over the class $\bigcup_{\tau \in \mathcal I_1, C \in \mathcal I_2}\set(\tau,b,C,C')$.
Then we can express
$\rho_n(\beta_{i}) \sim (\log(9M_n/\alpha))^{3/2} \log(n) n^{-\frac{\beta_{i}}{2\beta_{i}+1}}.$

Similar to the two points test,  we define succesive tests for $i = 0, \ldots, M_n-2$
 $$\Psi_n(i) = \mathbbm{1} \left\{T_n(i) \geq \frac{\rho_n(\beta_{i+2})}{2} \right\},$$
where
\begin{align*}
T_n(i) &= \sup_{x \leq \hat B} \left(|x^{\hat \tau}\hat p_x - \hat C|- C' x^{-\hat \tau \beta_i} \right)\\
\hat B &= n^{1/{\hat \vartheta(\beta_{i+2})}}\\
\hat \vartheta(\beta_{i+2}) &= \left(\hat \tau + n^{-\frac{\beta_{i+2}}{2\beta_{i+2}+1}}c_1(\alpha/(9M_n)) \right) (2\beta_{i+2}+1).
\end{align*}
By Theorem~\ref{thm:test2}, we know that the $i^{\text{th}}$ test is $(\alpha/M_n)$-uniformly consistent on $\set(\tau,\beta_{i},C,C') \bigcup \tilde \set(\tau,\beta_{i+2},\beta_i,C,C' \rho_n(\beta_{i+2}))$.

Let $F\in \mathcal P_n$ (for the $\mathcal P_n$ defined in Equation~\eqref{newmodel}), and let $\beta^*:=\beta^*(F)$ be the associated index defined in (\ref{betastar}). Let $i^*$ be the smallest index $i$ (corresponding to the largest $\beta_i$) such that $F$ is contained in $\mathcal{S}(\tau, \beta_i, C, C')$. That is, we have $\beta_{i^*} \leq \beta^*< \beta_{i^*-1}$ (and set by convention $i^* = 0$ if $\beta^* > \beta_{0}$). If $i^* \neq 0$, this implies that $F \in \set(\tau,\beta_{i^*},C,C')$ and $F \not\in \set(\tau,\beta_{i^*-1},C,C')$. For $F\in \mathcal P_n$, either (if $i^*=0$) $F\in \set(\tau, B,C,C')$ or (if $i^*\neq 0$) there exists $i\in \{0,\ldots, M_n-2\}$ such that $F \in \tilde \set(\tau,\beta_{i+2}, \beta_i,C,C', \rho_n(\beta_{i+2}))$.

Now we define the estimator $\hat i$ for $i^*$ by choosing two plus the maximum index (corresponding to the smallest $\beta$) which rejects the $i^{\text{th}}$ null hypothesis:     
\begin{equation}\label{istar}
\hat i =\max \Big\{i\in \{0,\ldots, M_n-2\}: \Psi_n(i) =1 \Big\} + 2, 
\end{equation}
and we set $\hat i = 0$ by convention if the set $\{i\in \{0,\ldots, M_n-2\}: \Psi_n(i) =1\}$ is empty.

\textbf{Case 1: $i^*\neq 0$.} Since, as mentioned above, we have $F \in \set(\tau,\beta_{i^*},C,C')$ and $F \not\in \set(\tau,\beta_{i^*-1},C,C')$, we know that either we have 
$$F \in \tilde \set(\tau,\beta_{i^*}, \beta_{i^*-2},C,C', \rho_n(\beta_{i^*-2})),$$ or
$$F \in \tilde \set(\tau,\beta_{i^*+1}, \beta_{i^*-1},C,C', \rho_n(\beta_{i^*-1})).$$
This implies by Theorem~\ref{thm:test2} that either $(i^*-2)^{\text{th}}$ test or $(i^*-1)^{\text{th}}$ test is $\frac{4\alpha}{9M_n}$-consistent. More precisely, with probability larger than $1-\frac{4\alpha}{9M_n}$, either $\Psi_n(i^*-2) = 1$ or $\Psi_n(i^*-1) = 1$.

Moreover, for any $i\geq i^*$, we know that 
$$F\in \set(\tau, \beta_i,C,C').$$
This implies by Theorem~\ref{thm:test2} that for any $M_n-2 \geq i \geq i^*$, $i^{\text{th}}$ test is $\frac{4\alpha}{9M_n}$-consistent. More presicely, for any $M_n-2 \geq i \geq i^*$, with probability larger than $1-\frac{4\alpha}{9M_n}$, $\Psi_n(i) = 0$ (see the proof of Theorem~\ref{thm:test2}).

By an union bound and by the definition of $\hat i$, with probability larger than $1-\alpha$, either $\hat i = i^*$ or $\hat i = i^*+1$.

\textbf{Case 2: $i^* = 0$.} In this case, $F\in \set(\tau, B,C,C')$. So for any $i\geq 0$, we know that 
$$F\in \set(\tau, \beta_i,C,C').$$
This implies by Theorem~\ref{thm:test2} that for any $M_n-2 \geq i \geq 0$, test $i$ is $\frac{4\alpha}{9M_n}$-consistent. More presicely, for any $M_n-2 \geq i \geq i^*$, with probability larger than $1-\frac{4\alpha}{9M_n}$, we have $\Psi_n(i) = 0$. By an union bound and the definition of $\hat i$, with probability larger than $1-\alpha$, $\hat i = i^* = 0$.

\textbf{Conclusion.}

From the two previous cases, we deduce that with probability larger than $1-\alpha$, we have
$|\hat i - i^*|\leq 1.$
This implies in particular that with probability larger than $1-\alpha$,
$|\beta_{\hat i} - \beta_{i^*}|\leq \frac{B-b}{M_n}.$
Moreover, since we have $|\beta^* - \beta_{i^*}| \leq \frac{B-b}{M_n}$ by definiton of $i^*$, together with the previous equation, we have that with probability larger than $1-\alpha$,
$$|\beta_{\hat i} - \beta^*|\leq 2\frac{B-b}{M_n}\leq 4\frac{\xi(B-b)}{\log (n)},$$
for $n$ large enough so that $\log(n) \geq 2\xi$. The constants in the bound are independent of $F$ so the result holds uniformly over $\mathcal P_n$.


\subsection{Proof of Theorem~\ref{generaltesting2}}\label{proof:generaltesting2}

Let $\alpha>0$ and let us consider the estimate $\hat i$ of $i^*$ defined in (\ref{istar}).

Consider the estimator $\hat \tau(\beta_{\hat i+1})$ for $\tau$ (e.g.~Hill's estimator or adaptive estimator as described before) using the sample fraction corresponding to $\beta_{\hat i+1}$. Then we define the confidence interval
$$
C_n = \left\{\tau' : | \hat \tau(\beta_{\hat i+1}) - \tau'| \leq G(\beta_{\hat i+1}) \right\},
$$
where $G(\beta) :=  c_1(\alpha/(12M_n)) n^{-\beta/(2\beta+1)}$ .

From the same argument in the previous Subsection \ref{proof:generaltesting}, we know that $|\hat i - i^*|\leq 1$ and $|\beta_{\hat i} - \beta^*|\leq 4\frac{\xi(B-b)}{\log (n)}$ with probability $1-3\alpha/4$. For notational convenience, we write $\set(\beta) := \bigcup_{\tau \in \mathcal I_1, C\in \mathcal I_2}\set(\tau, \beta, C, C')$.

First, we show that the constructed confidence interval $C_n$ is adaptive. By definition of $\beta^*$, and using $\beta^*(F) < \beta_{i^*-1} = \beta_{i^*}+(B-b)/M_n$, 
 we have (with notation $\varrho:=\frac{12(B-b)\xi}{(2b+1)^2}$)
\begin{align*}
&\sup_{\beta \in [b,B]} \sup_{ F \in \set(\beta) \cap \mathcal{P}_n} \pp_F \left(|C_n| > 2e^{\varrho}G(\beta) \right)\\
&\leq \sup_{F \in \mathcal{P}_n} \pp_F \left(|C_n| > 2e^{\varrho}G(\beta^*(F)) \right) \leq \sup_{F \in \mathcal{P}_n} \pp_F \left(|C_n| > 2e^{\varrho}G(\beta_{i^*-1}) \right)\\
&\leq \sup_{ F \in \mathcal{P}_n} \left\{ \pp_F \left( |C_n| > 2e^{\varrho}G \Big(\beta_{\hat i}+2\frac{(B-b)}{M_n}\Big), |\hat i - i^*|\leq 1\right) + \pp_F(|\hat i - i^*|> 1) \right\} \\
&\leq \sup_{ F \in \mathcal{P}_n} \left\{ \pp_F \left( |C_n| > 2e^{\varrho}G \Big(\beta_{\hat i+1}+3\frac{(B-b)}{M_n}\Big), |\hat i - i^*|\leq 1\right) + \frac{3\alpha}{4} \right\} \\
&\leq \sup_{ F \in \mathcal{P}_n} \left\{ \pp_F \Big( |C_n| > 2G(\beta_{\hat i+1})\Big) + \frac{3\alpha}{4}  \right\} \leq \frac{3\alpha}{4},
\end{align*}
where the penultimate inequality follows by (for $n$ large enough so that $\log(n) \geq2\xi$)
$$\frac{G\big(\beta_{\hat i} + 3\frac{(B-b)}{M_n}\big)}{G(\beta_{\hat i})} \geq n^{-\frac{6(B-b)}{M_n(2\beta_{\hat i +1}+1)^2}} \geq \exp(-\varrho),
$$
 and the last inequality is obtained since $|C_n| = 2G(\beta_{\hat i+1})$ and $|\hat i - i^*| \leq 1$ with probability $1-\frac{3\alpha}{4}$. 
Thus the confidence interval is adaptive.

Then it suffices to prove that the confidence interval is uniform. By definition of $C_n$ and since $|\hat i - i^*| \leq 1$ with probability $1-\frac{3\alpha}{4}$,
\begin{align*}
&\inf_{F \in \mathcal{P}_n} \pp_F(\tau \in C_n) = \inf_{F \in \mathcal{P}_n} \pp_F(|\hat \tau(\beta_{\hat i+1}) - \tau| \leq G(\beta_{\hat i+1}))\\
&\geq \inf_{F \in \mathcal{P}_n} \pp_F(|\hat \tau(\beta_{\hat i+1}) - \tau| \leq G(\beta_{\hat i+1}), |\hat i - i^*| \leq 1)\\
&= \inf_{F \in \mathcal{P}_n} \left\{1-\frac{3\alpha}{4}- \pp_F \left(| \hat \tau(\beta_{\hat i+1})-\tau| > G(\beta_{\hat i+1}), |\hat i - i^*| \leq 1\right)\right\} =(*)
\end{align*}
Now, we consider three possible cases for $\hat i \in \{i^*-1, i^*, i^*+1 \}$, which gives
\begin{align*}
(*) 
&= \inf_{F \in \mathcal{P}_n} \Big\{1-3\alpha/4- \pp_F \left(| \hat \tau(\beta_{i^*})-\tau| > G(\beta_{i^*})| \hat i = i^*- 1\right)\pp_F(\hat i = i^*- 1)\\
&- \pp_F \left(| \hat \tau(\beta_{i^*+1})-\tau| > G(\beta_{i^*+1})| \hat i = i^*\right)\pp_F(\hat i = i^*)\\
&- \pp_F \left(| \hat \tau(\beta_{i^*+2})-\tau| > G(\beta_{i^*+2})| \hat i = i^*+1\right) \pp_F(\hat i = i^*+ 1) \Big\}\\
&\geq \inf_{F \in \mathcal{P}_n} \Big\{1-3\alpha/4- \\
&\ \ \ \ \ \  \Big(\sum_{j \in \{-1,0,1\}}\pp_F \big(|\hat \tau(\beta_{i^*-j})-\tau| \geq G(\beta_{i^*-j})\big)\Big) \pp_F(|\hat i - i^*| \leq 1)\Big\}\\
&\geq 1-\frac{3\alpha}{4} - 3 \Big(\frac{\alpha}{12M_n}\Big)\Big(1-\frac{3\alpha}{4}\Big)  \geq 1-\alpha,
\end{align*}
where the last inequality follows from the definition of $c_1(\alpha/(12M_n))$, that is, $c_1(\epsilon)$ is chosen such that $|\hat \tau (\beta) - \tau| \leq c_1(\epsilon)n^{-\beta/(2\beta+1)}$ with probability $1-\epsilon$ (for $F \in \set(\beta)$). This concludes the proof.

\subsection{Proof of Lemma~\ref{lem:self}}\label{proof:self}


Let $F\in \mathcal G$, then for any $x \geq D$, it verifies
$C_1x^{-\tau(\beta+1)} \leq |1-F(x) - Cx^{-\tau}| \leq C'x^{-\tau(\beta+1)}$ by (\ref{eq:self}).
Note that $\beta^*(F) = \beta \leq  B - 2 \frac{(B-b)}{M_n}$ by definition of $\mathcal{G}$.

Let $i^*$ be defined as in the paragraph above (\ref{istar}). By definition of $i^*$, we know that $\beta_{i^*}\leq \beta < \beta_{i^*-1} = \beta_{i^*} + \frac{B-b}{M_n}$. Moreover, since $\beta \leq  B - 2 \frac{(B-b)}{M_n}$, we know that $i^* \geq 2$.

By definition, we know that $F \in \mathcal S(\tau, \beta_{i^*}, C, C')$. Moreover, we have
\begin{align}
\|F - \mathcal S&(\tau, \beta_{i^*-2}, C, C')\|_{\tau, \infty} = \sup_{F_0 \in \mathcal S(\tau, \beta_{i^*-2}, C, C')} \sup_x |x^{\tau}F_0(x) - x^{\tau}F(x)| \nonumber \\
&\geq \sup_{F_0 \in \mathcal S(\tau, \beta_{i^*-2}, C, C')} \sup_{x \geq D}\Big( |x^{\tau}(1-F(x))-C| - |x^{\tau}(1-F_0(x))-C|\Big) \nonumber\\
&\geq  \sup_{x \geq D}\Big( |x^{\tau}(1-F(x))-C| - \sup_{F_0 \in \mathcal S(\tau, \beta_{i^*-2}, C, C')} |x^{\tau}(1-F_0(x))-C|\Big) \nonumber\\
&\geq  \sup_{x \geq D}\Big( C_1 x^{-\tau\beta} - C'x^{-\tau\beta_{i^*-2}} \Big)
=  \sup_{x \geq D}\Big(x^{-\tau\beta} \big( C_1 - C'x^{-\tau(\beta_{i^*-2}- \beta)}\big) \Big)\nonumber\\
&\geq  \sup_{x \geq D}\Big(x^{-\tau\beta} \big( C_1 - C'x^{-\tau(B-b)/M_n}\big) \Big), \label{lastequation}
\end{align}
since $\beta < \beta_{i^*} + \frac{B-b}{M_n}$ and $\beta_{i^*-2}-\beta_{i^*}>0$.

Now let us consider $x:= \Big(\frac{n}{(\log(n)^2\log\log(n)^{3/2})^{\frac{2\beta+1}{\beta}}}\Big)^{\frac{1}{\tau(2\beta+1)}}$. For $n$ large enough, it is larger than $D$. The equation (\ref{lastequation}) implies  that we have 
\begin{align*}
\|F - \mathcal S&(\tau, \beta_{i^*-2}, C, C')\|_{\tau, \infty} \geq \log(n)^2 \log\log(n)^{\frac{3}{2}}n^{-\frac{\beta}{2\beta+1}} \big( C_1 - C'x^{-\frac{\tau\xi (B-b)}{2\log(n)}}\big).
\end{align*}
Then we can upper bound $x^{-\frac{\tau\xi (B-b)}{2\log(n)}}$ for $n$ large enough such that $\log\log(n)^2 \leq 2\log(n)$, 
\begin{align*}
 x^{-\frac{\tau\xi (B-b)}{2\log(n)}} &= e^{-\frac{\xi(B-b)}{2(2\beta+1)}} \left((\log(n))^2(\log \log(n))^{\frac{3}{2}} \right)^{\frac{\xi(B-b)}{2\beta(\log n)}}\\
&\leq e^{-\frac{\xi(B-b)}{2(2\beta+1)}} \left(2^{\frac{3}{2}} (\log(n))^{3}\right)^{\frac{\xi(B-b)}{2\beta(\log n)}}\leq 2e^{-\frac{\xi(B-b)}{2(2B+1)}},
\end{align*}
where the last inequality follows by choosing $n$ large enough so that 
$$\frac{\xi(B-b)}{2\beta \log(n)} \log \left(2^{\frac{3}{2}}(\log(n))^3 \right) \leq 1/2.$$
Also, we bound using $\beta_{i*} \leq \beta < \beta_{i^*}+(B-b)/M_n$,
$$
n^{-\frac{\beta}{2\beta+1}}=n^{-1+\frac{1}{2\beta+1}} > n^{-1 + \frac{1}{2(\beta_{i^*}+ (B-b)/M_n)+1}}\geq n^{-\frac{\beta_{i^*}}{2\beta_{i^*}+1}}
e^{-\frac{\xi(B-b)}{2(2\beta+1)}}.
$$
Finally
\begin{align*}
\|F - \mathcal S(\tau, \beta_{i^*-2}, C, C')\|_{\tau, \infty} &\geq  \log(n)\log\log(n)^{3/2}n^{-\frac{\beta_{i^*}}{2\beta_{i^*}+1}} \\\
&\times \log(n)e^{-\frac{\xi(B-b)}{2(2\beta+1)}}\left(C_1-2C'e^{-\frac{\xi(B-b)}{2(2B+1)}}\right)
\end{align*}
%
Since $C_1-2C'e^{-\frac{\xi(B-b)}{2(2B+1)}}$ is positive by assumption, $\log(n)\big(C_1-2C'  e^{-\frac{\xi(B-b)}{2(2B+1)}}\big)$ diverges to infinity when $n$ goes to infinity. This implies in particular that
\begin{align*}
\log(n)^{-1}\log\log(n)^{-3/2}n^{\frac{\beta_{i^*}}{2\beta_{i^*}+1}} \|F - \mathcal S(\tau, \beta_{i^*-2}, C, C')\|_{\tau, \infty} \rightarrow  \infty,
\end{align*}
and $F \in \tilde{\mathcal{S}}(\tau, \beta_{i^*}, \beta_{i^*-2}, C, C',\rho_n(\beta_{i^*}))$. This concludes the proof.

\subsection{Proof of Lemma~\ref{lem:discdist}}\label{proof:selfi}

Since $F \in\mathcal H$, we have
$$1 - F(x) = Cx^{-\tau} + \tilde C x^{-\tau (1+\beta)}+ o(x^{-\tau (1+\beta)}).$$
Let $0\leq u<t$ and $k \in \{a_j,j\in \mathbb N\}$ be large compared to $t$. We have by definition of $\tilde F$ followed by Taylor expansion,
\begin{align}
1 - &\tilde F(k + u) = 1 - F(k) = C k^{-\tau}+ \tilde C k^{-\tau(1+\beta)} + o(k^{-\tau(1+\beta)}) \nonumber\\
&= C (k+u)^{-\tau} (1-\frac{u}{k+u})^{-\tau} + \tilde C (k+u)^{-\tau(1+\beta)} + o( (k+u)^{-\tau(1+\beta)}) \nonumber\\
&= C (k+u)^{-\tau} (1+\frac{ u\tau}{k+u}) +o\big((k+u)^{-\tau(1+1/\tau)}\big) + \tilde C (k+u)^{-\tau(1+\beta)} \nonumber\\ 
&\ \ \ + o\big((k+u)^{-\tau(1+\beta)}\big) \nonumber\\
&= C (k+u)^{-\tau} + u \tau (k+u)^{-\tau(1+1/\tau)} + \tilde C (k+u)^{-\tau(1+\beta)} \nonumber\\ 
&+ o\big( (k+u)^{-\tau(1+\min(\beta, 1/\tau))} \big).\label{eq:mimite}
\end{align}
This implies in particular that $\beta^*(\tilde F) = \min(\beta, 1/\tau)$, and that 
$$\tilde F \in \mathcal{S}(\tau, \min(\beta, 1/\tau), C, C'),$$
where $C'$ is a large enough constant.

\paragraph{Claim 1: $\tilde F \not \in \mathcal H$ if $\beta>1/\tau$.} 

Assume that $\beta>1/\tau$. In this case, $\beta^*(\tilde F) = 1/\tau$ Then by Equation~\eqref{eq:mimite}
\begin{align*}
1 - \tilde F(x) = C x^{-\tau} + (x - k) \tau x^{-\tau(1+\beta^*(\tilde F))} + o( x^{-\tau(1+\beta^*(\tilde F))} ),
\end{align*}
where $k$ is the largest element of $\{a_j, j \in \mathbb N\}$ that is smaller than or equal to $x$.

Consider first the case where $x = k \in\{a_j, j \in \mathbb N\}$. Then $u=0$, so 
$$1- \tilde F(x) = Cx^{-\tau} + o( x^{-\tau(1+\beta^*(\tilde F))} ) = x^{-\tau} \big(C + o( x^{-\tau\beta^*(\tilde F)} )\big),$$ which shows that the second order term in (\ref{exachall}) must be $0$.

Consider now $x = k+t/2$ where $k\in\{a_j, j \in \mathbb N\}$. We have
$$1- \tilde F(x) = C x^{-\tau} + \frac{t\tau}{2} x^{-\tau(1+\beta^*(\tilde F))} + o( x^{-\tau(1+\beta^*(\tilde F))} ),$$ 
which shows that the second order term must be of the form $(t\tau/2)x^{-\tau(1+\beta^*(\tilde F))}$.

The two previous equations imply that $\tilde F$ cannot be in model~\eqref{exachall} (and thus also not in the model~\eqref{comp:model1}).

\paragraph{Claim 2: $\tilde F \in \mathcal P_n$.} 

Let $b,B$ be such that $\beta^*(\tilde F) = \min(\beta, 1/\tau) \in (b,B)$. Let $\epsilon>0$. For $x$ large enough, we have by Equation~\eqref{eq:mimite}
\begin{align}
|1- \tilde F(x) - Cx^{-\tau} - &(x-k) \tau x^{-\tau(1+1/\tau)} - \tilde C x^{-\tau(1+\beta)}\Big| \nonumber \\ 
&\leq  \min(\tau t/4, |\tilde C|/2)x^{-\tau(1+\beta^*(\tilde F))}.\label{eq:mimite2}
\end{align}
Let $D$ be the constant such that for any $x\geq D$, the above equation is satisfied.



Let $i^*$ be defined as in the paragraph above in (\ref{istar}), and $n$ be large enough so that $i^* \geq 2$ (which is possible since $1/\tau \in (b,B)$). We will prove $\tilde F \in \mathcal P_n$ by showing $\tilde F \in \tilde{\mathcal{S}}(\tau, \beta_{i^*}, \beta_{i^*-2}, C, C', \rho_n(\beta_{i^*}))$ for $n$ large enough.  

By definition, we know $\tilde F \in \mathcal S(\tau, \beta_{i^*}, C, C')$. As in the proof of Lemma~\ref{proof:self}, we have for $n$ large enough, by Equation~\eqref{eq:mimite2}
\begin{align}
\|\tilde F - &\mathcal S(\tau, \beta_{i^*-2}, 1, C')\|_{\tau, \infty} = \sup_{F_0 \in \mathcal S(\tau, \beta_{i^*-2}, C,  C')} \sup_{x\geq D} |x^{\tau}F_0(x) - x^{\tau}\tilde F(x)| \nonumber \\
&\geq \sup_{F_0 \in \mathcal S(\tau, \beta_{i^*-2}, C, C')} \sup_{x\geq D}\Big( |x^{\tau}(1-\tilde F(x)) - C| - |x^{\tau}(1-F_0(x)) - C|\Big) \nonumber\\
&\geq  \sup_{x\geq D}\Big( \min_{|s|\leq 1} \Big|(x-k) \tau x^{-\tau/\tau} + \tilde C x^{-\tau\beta} + s\min(\tau t/4, |\tilde C|/2)x^{-\tau\min(\beta, 1/\tau)}\Big| \nonumber\\ 
&- \sup_{F_0 \in \mathcal S(\tau, \beta_{i^*-2}, C, C')} |x^{\tau}(1-F_0(x))-C|\Big) \nonumber\\
&\geq  \sup_{x \geq D, \ x \in \{a_j,j\in\mathbb N\}+t/2}\Big(  \min(|\tilde C|/2, \tau t/4) x^{-\tau(\min(\beta, 1/\tau))} \label{eq:hellio}\\ 
&\ \ \ \ - \sup_{F_0 \in \mathcal S(\tau, \beta_{i^*-2}, C, C')} |x^{\tau}(1-F_0(x))-C|\Big) \nonumber\\
&\geq  \sup_{x \geq D, \ x \in \{a_j,j\in\mathbb N\}+t/2}\Big(  C_\tau x^{-\tau \beta^*(\tilde F)} - C'x^{-\tau\beta_{i^*-2}} \Big) \nonumber\\
&=  \sup_{x \geq D, \ x \in \{a_j,j\in\mathbb N\}+t/2}\Big(x^{-\tau \beta^*(\tilde F)} \big( C_\tau - C'x^{-\tau(\beta_{i^*-2}- \beta^*(\tilde F))}\big) \Big)\nonumber\\
&\geq  \sup_{x \geq D, \ x \in \{a_j,j\in\mathbb N\}+t/2}\Big(x^{-\tau \beta^*(\tilde F)} \big( C_\tau - C'x^{-\tau(B-b)/M_n}\big) \Big), \label{lastequation2}
\end{align}
where Equation~\eqref{eq:hellio} comes from the fact that for $x \in \{a_j,j\in\mathbb N\}+t/2$, we have
\begin{align*}
1 - \tilde F(x) = C x^{-\tau} +  \frac{t}{2} \tau x^{-\tau(1+1/\tau)} + \tilde C x^{-\tau(1+\beta)} + o( x^{-\tau(1+\min(\beta,1/\tau))} ),
\end{align*}
and where $C_\tau = \min(|\tilde C|/2, \tau t/4)$, and (\ref{lastequation2}) follows
since $\beta^*(F) < \beta_{i^*} + \frac{B-b}{M_n}$ and $\beta_{i^*-2}-\beta_{i^*}>0$. By the same bounds as in $\S$~\ref{proof:self} on $C'x^{-\tau(B-b)/M_n}$ (for a constant $\xi$ large enough), it follows 
$$F \in \tilde{\mathcal{S}}(\tau, \beta_{i^*}, \beta_{i^*-2},C, C',\rho_n(\beta_{i^*})).
$$
This concludes the proof by definition of $\mathcal P_n$.


\vspace{10pt}
\textbf{Acknowledgements} \newline
The authors are grateful to Richard Nickl and Richard Samworth for helpful suggestions and valuable advice.

\bibliographystyle{chicago}
\bibliography{arlene_alex}

\end{document}